\documentclass{amsart}
\usepackage{graphicx}
\newtheorem{theorem}{Theorem}[section]
\newtheorem{lemma}[theorem]{Lemma}
\newtheorem{proposition}[theorem]{Proposition}
\theoremstyle{definition}

\newtheorem{corollary}[theorem]{Corollary}

\theoremstyle{remark}

\numberwithin{equation}{section}

\def \R{{\bf R}}
\def \N{{\bf N}}

\def \P{{\bf P}}
\def\E{{\bf E}}
\def \C{{\bf C}}

\def \i{{\infty}}

\def\z{{\zeta}}
\def\l{{\lambda}}
\newcommand {\cc}{\check}
\begin{document}

\title{Explicit solutions for the exit problem for a class of L\'evy processes. Applications to the pricing of double barrier options}

\author{Sonia FOURATI}
\address{INSA de ROUEN,LMI  et LPMA de Paris VI et VII}
\curraddr{ 76801 Saint Etienne du Rouvray, France
France}
\email{sonia.fourati@upmc.fr}


\date{February 2010}


\begin{abstract}
Lewis  and Mordecki  have computed  the Wiener-Hopf factorization of a L\'evy process  whose restriction on $]0,+\i[$ of their   L\'evy measure has a rational Laplace transform. That allows to compute  the  distribution of $\displaystyle  (X_t,\inf_{0\leq s\leq t}X_s)$. For the same class of L\'evy processes, we compute the distribution of 
$\displaystyle  (X_t,\inf_{0\leq s\leq t}X_s,\sup_{0\leq s\leq t} X_s)$ and also the behavior of this triple at certain stopping time, like the first exit time of an interval containing the origin. Some applications to the pricing of  double barrier options with or without rebate are evocated. 
\end{abstract}

\maketitle

\section{Introduction} There are very few  examples of L\'evy process for which the so-called "exit problem" can be explicitly solved (see \cite  {KK05} , \cite {R90}, \cite  {R72}). We present here this  explicit solution for a class of  L\'evy processes which has been introduced by Lewis A.L. and 
Mordecki E.\cite {LM08}; that is the class  of  the  L\'evy processes whose restriction on $]0,+\i[$ of their   L\'evy measure has a rational Laplace transform. This happens when this restriction is a finite linear combination of exponential or gamma distributions. 

Lewis A.L and Mordecki E. \cite  {LM08} (see also Asmussen, S. Avram, F. and Pistorius M.R. \cite {AAP04})  have computed  the so-called Wiener-Hopf factorization of these  L\'evy process. That brings to closed forms for  the distribution of the maximum of the process before an independent exponential time and for  the joint distribution of $(T^x, X_{T x})$  where $T^x$ is the first time where the L\'evy process $X$ crosses upward  a level $x$ and $X_{T^x}$  its position at that time. 

As an application, when adopting the exponential L\'evy model, $Y_t=Y_0.e^{X_t}$,  for a financial asset $Y_t$  with $X_t$ of the preceding form (see for example \cite {CT03}),  an immediate consequence   of the preceding results  is  the computation of the (temporal Laplace transform)  price of the double barrier option with this underlying asset. 
That is, an  european option which  is activated (in) or  desactivated (out) when the asset $(Y_t)_{t\in [0,T]}$ cross up (down)  a barrier $H$ before the time of maturity  $T$. This application follows the computation of the price of the simple barrier option, by   the Wiener-Hopf factorization (see \cite {AAP04}).

In this paper, for the same class of  L\'evy processes, we solve the exit problem. More precisely, we give closed form of the joint distribution of the minimum and  maximum of the process before an independent exponential time and among other behaviors of the process at certain stopping times, we give the Laplace transform of the joint distribution of $(T_a^b, X_{T_a^b})$ where  $T_a^b$ is the first  time that the process  leaves  a bounded interval $[-a,b]$ containing the origin  and $X_{T_a^b}$ is  its  position at that time.

As an application, we mention how one can deduce the price of the double barrier option, in or out, that means that the option is activated or desactived if the asset crosses up a barrier $b$ or down a barrier $a$ in the exponential L\'evy model. 

The paper is organized as follow : In section 2, we recall the general results on Wiener-Hopf factorization, we give a shortened proof of the results  of Lewis-Mordecki and settle few other preliminary results. For that, we use  exclusively elementary complex analysis arguments. 
After introducing few more  notations in section 3, we give in section 4, 5, 6, 7 and 8,  all the results on the fluctuations of our  class of L\'evy processes. In section 9, we recall the main result of \cite{F10} which is the main tool of this work. The proofs follow in section 10, 11, 12, 13.

These results, in the symmetric case, are related to what is called "Bargmann equations"  in the litterature on inverse problems of the spectral theory (see for example \cite{F63}).

\section{The assumption and the Wiener-Hopf factorization}

We suppose that $X$ is a real L\'evy process possibly killed at an independent exponantial time and  we denote $\z$ the life time of $X$.

Let  $\phi$ be the  L\'evy exponent of $X$, so that  the identity 
$$\E(e^{-iu X_t}1_{t<\z})=e^{-t\phi(iu)}$$ 
is fulfilled for every time $t$ and every  imaginary number $iu \in i\R$, $\phi$ is continuous on $i\R$ and  $\phi(0)$ is the rate of the exponential distribution of the life time $\z$, 
$$\phi(0)=0\quad \hbox{  if and only if  $\z=+\i$ a.s. ("$X$ does not die")}. $$

We now work under Lewis-Mordecki's assumption. It is based on the following rather obvious fact. 

\begin{proposition} \label{Ass} The conditions below  are equivalent 

$(i)$ The L\'evy measure  of $X$, $\pi$, is of the following form 

 $$\pi (dy)=\sum_{j=1}^n c_j {y^{n_j}\over n_j!} e^{-\gamma_j y}dy \qquad \hbox{on}\quad  ]0,+\i[$$
 
 $$n\in \N, n_j\in \N ,  \Re(\gamma_j)>0,  c_j\in \C \qquad (j=1,\dots , n)  $$
 
 $(ii)$ The exponent $\phi$ is of the form $$\phi(iu)=\phi^-(iu)+{P\over Q}(iu),$$ where $\phi^-$ is the
 exponent of a L\'evy process without positive jumps,  
$P$ and $Q$ are polynomials and  $Q$ has   all its roots on the complex half plane $\{\l; \Re(\l)<0\}$.

\end{proposition}

\begin{proof} Assume $(i)$. Notice that the compound Poisson process with L\'evy measure $$\pi(dy)=\sum_{j=1}^n c_j {y^{n_j}\over n_j!} e^{-\gamma_j y}1_{y>0}dy$$ has L\'evy exponent $$\sum_{j\in J}  c_j [{1\over (\gamma_j)^{n_j+1}}- {1\over (\l+\gamma_j)^{n_j+1}}]=:{P(\l)\over Q(\l)},$$ 

The function  $\phi-{P\over Q}$ is then the L\'evy exponent of a L\'evy process without positive jumps. This establishes $(ii)$.

Conservely, assume $(ii)$. Since $\phi$ and $\phi^-$  are    exponents of  L\'evy processes, we know  that  $\phi(iu)=O(u^2)$ and  $\phi^-(iu)=O(u^2)$ (see proposition  2 chapter 1 of \cite {B96}), thus ${P\over Q}(iu)$, construed as a rational function of $u$, is $O(u^2)$. Consequently $\deg P\leq \deg Q+2$.

Now, write ${P\over Q}(iu)$ in its fractional  expansion
$${P\over Q}(iu)= au^2+ b u +c+  \sum_{j=1}^n {k_j\over (iu+\gamma_j)^{n_j+1}}$$

Then ${P\over Q}(iu)$ is the Fourier transform of the Schwartz distribution (here, $\delta$ stands for the Dirac mass at $0$); 

$$a\delta'' -ib\delta' +c \delta + \sum_{j=1}^n  k_j {x^{n_j}\over n_j!}e^{-\gamma_j x} 1_{x>0}dx$$

On the other hand, one can deduce from the  L\'evy Kinchin formula applied to $\phi$  that $\phi(iu)$ is the Fourier transform of a distribution whose restriction to $]0,+\i[$ is $-1_{x>0}\pi(dx)$ (see chapter 5 of \cite {V02} for example) where $\pi$ is the L\'evy measure. From this, we see that  
$$1_{x>0}\pi(dx)=-\sum_{j=1}^n k_j {x^{n_j}\over n_j!}e^{-\gamma_j x} 1_{x>0}dx$$

This finishes the proof.

\end{proof}

The following is a famous result (see proposition 2, chapter I of \cite{B96})  which applies to any L\'evy process and  which will be  of use later .

\begin{lemma} There  exists an exponent of  a subordinator $\psi$ and an exponent of the opposite of a subordinator  $\cc \psi$, such that  

$$\psi.\cc \psi=\phi\qquad \hbox{on}\qquad  i\R$$

Such a couple $(\psi,\cc \psi)$ is unique up to a multiplicative constant.

\end{lemma}

The functions $\psi$ and $\cc \psi$  will be refered to as the "positive" (for $\psi$)  and "negative"  (for $\cc\psi$) Wiener -Hopf factors of the exponent $\phi$, unlike the ordinary uses  which affect these expressions to the functions ${\psi(0)\over\psi}$ and ${\cc\psi(0)\over\cc\psi}$.

{\bf ASSUMPTION  :  In the rest of this paper, conditions  $(i)$ and $(ii)$ of proposition  \ref{Ass} will be assumed to be satisfied and this assumption will be call Assumption \ref{Ass} }.

Property $(ii)$ of proposition \ref{Ass} and the fact that  $\phi^-$ has an holomorphic extension on the half plane $\{\Re(\l)<0 \}$ (see chapter 7 of \cite{B96}) imply that $\phi$ has a meromorphic extension on this half plane, and we  we will still denote by $\phi$ this extension.

We will denote $$-\gamma_1,-\gamma_2,...,-\gamma_n$$  the poles of $\phi$ repeated according to their multiplicity.  In particular, $\deg  Q=n$ when the fraction  ${P\over Q}$  is irreductible, in the expression of condition $(ii)$ of proposition \ref{Ass} ($\phi=\phi^-+{P\over Q}$) and we will denote 

$$n_j: =\sharp\{ k<j; \gamma_k=\gamma_j\}\qquad j=1,\dots,n $$ 

Notice that the $n_j$ are all zero if and only if  all the poles of $\phi$ on  the half plane $\{Re(\l)<0\}$ are simple, or equivalently, if and only if the L\'evy measure is a (not necessarily positive) combination of exponential distributions.

\begin{theorem} \label{WH}
$\psi$ is of the form :
$$\psi(\l)=\psi_{\i}{\Pi_1^m(\l+\beta_i)\over \Pi_1^n(\l+\gamma_j)},$$

where $\psi_{\i}$ is a positive constant. 

If $\lim_{\l\to -\i}{\phi(\l)\over \l}\in ]0,+\i]$ then $m=n+1$.  Otherwise, $\lim_{\l\to -\i}{\phi(\l)\over \l}\in ]-\i,0]$ and $m=n$. 

More over, $\{-\beta_1,\dots,-\beta_n\}$ is the set of the roots of $\phi$ lying on the half plane $\{\Re(\l)<0\}$ (repeated according to their multiplicity)  together  with $0$ if $\phi(0)=0$ and $\phi'(0)\geq 0$. 
 \end{theorem}

\begin{proof}  Notice first that, thanks to L\'evy Kintchin formula, $\phi(\l)$ does not vanish for $\l$ on the imaginary axis except possibly  for $\l=0$,  
$\phi$ is a meromorphic funtion on $\{\Re(\l)<0\}$ and $\phi$ is continuous by the left at each point of the imaginary axis $i\R$.
 
 On the other hand, the function $\psi$ is an exponent of a subordinator, thus it is analytic on the open half plane $\{\Re(\l)>0\}$and continuous on the closed half plane $\{\Re(\l)\geq 0\}$ (see \cite {B96} chapter 3  for example).
 Also, the  identity $\psi\cc\psi =\phi$ on $i\R$ implies that  $\psi={\phi\over \cc\psi}$ on $i\R$. Thus, since 
 $\phi$ is meromorphic and $\cc\psi$ holomorphic on $\{\Re(\l)<0\}$, we see that $\psi$ has a meromorphic extension on $\{\Re(\l)<0\}$.We  denote  again $\psi$  this extension. 

More over $\psi={\phi\over\cc\psi}$ is continuous by the left on the axis $i\R$, except possibly at $\l=0$ when $\cc\psi(\l)=0$.

Then $\psi$ is meromorphic on $\C$ except possibly at $0$.   
However, because $\cc\psi$ is the exponent of the opposite of a subordinator, we have (see page 73 of \cite{B96}) $$\lim_{\l\to 0;\Re(\l)< 0}{\l\over\cc\psi(\l)}\to ]-\i,0],$$   then 
 $$\lim_{\l\to 0;\Re(\l)<0} \l.{\phi(\l)\over \cc\psi(\l)}=\lim_{\l\to 0;\Re(\l)<0}\l\psi(\l)=0.$$ 
 
 Also, $\psi$ is continuous by the right at the point $0$, then $\lim_{\l\to 0;\Re(\l)>0}\l\psi(\l)=0$ and we deduce that $\psi$ is again  holomorphic at $0$. 
 
 In conclusion, $ \psi$ is meromorphic on $\C$ and its poles, necessarily in the half plane $\{\l;\Re(\l)<0\}$,   are  the same as those of $\phi$ (because of  the identity $\psi\cc\psi=\phi$ and because $\cc\psi$ is holomorphic on $\{\l;\Re(\l)<0\}$). Therefore, these poles are the  $-\gamma_j$. 
 
 Now, write  $$\psi(\l)={\Psi(\l) \over \Pi_1^n(\l+\gamma_j)},$$ 
where  $\Psi$ is an entire function. Notice  the limit,  valid for any exponent of a subordinator ( see proposition 2 of Chapter 1 of \cite{B96})

$$\lim_{\l \to +\i}{\psi(\l)\over \l}\in [0,+\i[$$

Then 

\begin{equation} \label {0+}\Psi(\l)= O(\l^{n+1})\qquad (\Re(\l) \to +\i)\end{equation}

And,  since $$\psi (\l) ={\Psi(\l) \over \Pi_1^n(\l+\gamma_j)}= {\phi^-(\l)+{P(\l)\over Q(\l)} \over \cc\psi(\l)},$$ 
 $\phi^-(\l)$ and  ${P(\l)\over Q(\l)}$ are $O(\l^2)$  on $\{\Re(\l)<0\}$   (see proof of proposition \ref {Ass}) and   ${1\over\cc \psi(\l)}$ is bounded (by ${1\over\cc\psi(-1)}$) on $\{\Re(\l)<-1\}$,  because  it is the Laplace transform of a measure supported by $]-\i,0]$ (see chapter 3 of \cite {B96}  for example)  and $\deg P\leq \deg Q +2$ , we deduce that 
 
 \begin{equation} \label{0-}\Psi(\l)= O(\l^{n+2})\qquad (\Re(\l) \to -\i)\end{equation}
 
 When joining property (\ref{0+}) and (\ref{0-}), $\Psi$ is a polynomial with degree  at most $n+1$. Moreover, since ${1\over \psi(\l)}$ is bounded for $\l\to +\i$, then $\Psi$ has a degree at least $n$. 
 
 Thus 
 $\Psi$ is a polynomial and its degree is $n$ or $n+1$.
 
If $\Psi$ is a polynomial of degree $n+1$ then  
$$\lim_{\l\to-\i}{\psi(\l)\over \l}\in ]0,+\i[\quad\hbox{and}\quad \lim_{\l\to -\i}\cc \psi(\l)\in ]0,+\i]$$ 
(this last limit is valid for any exponent of the opposite of a subordinator). Since   $\phi=\psi\cc\psi$, we obtain  $$\lim_{\l\to -\i} {\phi(\l)\over \l}\in ]0,+\i].$$

And, if $\Psi$ is a polynomial of degree $n$,  then 
$$\lim_{\l\to-\i}\psi(\l)\in ]0,+\i[\quad\hbox{and}\quad \lim_{\l\to-\i}{\cc \psi(\l)\over \l}\in ]-\i,0],
$$ (this last limit is valid for any exponent of the opposite of a subordinator) and  we obtain $$\lim_{\l \to -\i } {\phi(\l)\over\l}\in ]-\i,0]$$

Finaly, $\cc \psi$ does not vanish on the left  half plane $\{\Re(\l)<0\}$ because it is the exponent of the opposite of a subordinator and so, the roots of $\psi$, which are necessary on the left half plane $\{Re(\l)<0\}$, except possibly $0$  (because the exponent of a subordinator doesn't vanish in the half plane $\{Re(\l)>0\}$), are  the one of $\phi$. Let us study when $0$ is a root of $\psi$. 

The function  $\cc \psi$ has a left derivative at $0$ (possibly equal to $-\i$) and the  function $\psi$ has a derivative at $0$ because it is a rational function and $0$ is not a pole. Then, we can write (here, the symbol $'$ has to be understood as a left derivative )
 
 $$\phi(0)=\psi(0)\cc\psi(0)$$ 
$$\phi'(0)=\cc\psi'(0).\psi(0)+ \cc\psi(0)\psi'(0).$$

Since $\psi(0)\in [0,+\i[$, $\cc\psi(0)\in[0,+\i[$, $\cc\psi'(0)\in [-\i,0[$, $\psi'(0)\in ]0,+\i]$ (that are obvious  consequences of L\'evy-Kinchin formula for exponent of subordinator or opposite of a subordinator), we obtain :

$$\phi(0)=0\Longrightarrow \biggl(\psi (0)=0\hbox{  and }\phi'(0)\geq 0\biggr) \hbox{ or }\biggl(\cc\psi(0)=0 \hbox{ and }\phi'(0)\leq 0\biggr)$$

And the proof is finished.
\end{proof}

{\bf In the sequel we will normalize $\psi$ by setting $\psi_{\i}=1$.}

We will always denote $$-\beta_1,-\beta_2,\dots,-\beta_m$$ the (possibly equal) roots of  $\psi$ and 
$$m_i=\sharp \{k<i; \beta_k=\beta_i\} $$

 ( $m_i=0$ for all $i$'s iff the roots $-\beta_1,\dots,-\beta_m$   of $\psi$ are simple).

{\bf Remark} If $X$ is of bounded variations and has a non positive drift then $m=n$. In all other cases, 
we have $m=n+1$, $\psi$ is then the exponent of a subordinator with positive drift and  $X$ "creeps upwards" (see theorem 19 chapter 6 of \cite{B96}).

Now, we introduce few notations which will be useful in the sequel, and which are related to the negative Wiener-Hopf factor $\cc \psi$.

First, $\cc \psi$ being the exponent of the opposite of a subordinator, the inverse of $\cc\psi$, ${1\over\cc\psi}$, is the Laplace transform of a measure supported by $]-\i,0]$. We will denote this measure   by $\cc U(dy)$, 

$${1\over \cc \psi(\l)}=:\int_{]-\i, 0]} e^{-\l y} \cc U(dy)\qquad (\Re(\l)<0)$$

Also,  for all $x\in ]0,+\i[$, we  denote   by $\cc U_{[-x,0]}(dy)$ and $\cc U_{]-\i,-x[}(dy)$ the measures $1_{[-x,0]}(y)\cc U(dy)$ and 
$1_{]-\i,-x[}(y)\cc U(dy)$ and by $\cc U_{[-x,0]}(\l)$ and $\cc U_{]-\i,-x[}(\l)$ their Laplace transform; 

$$\cc U_{[-x,0]}(\l):=\int_{[-x,0]} e^{-\l y}\cc U(dy)\quad (\l\in \C)$$
$$\cc U_{]-\i,-x[}(\l):=\int_{]-\i,-x[} e^{-\l y}\cc U(dy)\qquad (\Re(\l)<0)$$

Using standard facts about the relation between  negative  Wiener-Hopf factor $\cc \psi$ and the fluctuations of the L\'evy process (see  chapter 6 of \cite{B96}), we obtain  the next proposition.

\begin{proposition} \label{Flu-}

1)  If   $\phi(0)>0$ or $\phi'(0)>0$ then then the function ${\cc \psi(0)\over \cc\psi(\l)}$ is the Laplace transform of the measure $\cc\psi(0)\cc U(dy)$ which is  the distribution of $m$, the minimum of $X$ ($m:=\inf\{ X_t; 0\leq t<\z\}$).

2) $\cc \psi(\l)\cc U_{]-\i,-x[}(\l)$ is the Laplace transform of the distribution $\P(X_{T_x}\in dy; T_x<\z)$, ($T_x:=\inf\{t ; X_t<-x\}$).

\end{proposition}

Now, using the dual properties settled in this proposition which  involves the positive Wiener-Hopf factor $\psi$ instead of the negative one $\cc\psi$ and using  the explicit form of $\psi$ given in theorem \ref{WH}, we obtain, 
(see also Mordecki \cite{LM08} theorem 2.2),

\begin{corollary}\label{Flu+}

1) If  $\phi(0)>0$ or $\phi'(0)<0$ then  the distribution of the maximum, $M:=\{\sup; X_t; t\in [0,\z[\}$ is 

$$\P(M\in dy)=a_0\delta_0(dy) + \sum_{i=1}^m  a_i {y^{m_i}\over m_i!} e^{-\beta_i y} dy,$$

where the coefficients $a_i$ are given by the rational expansion :

$${\Pi_{j=1}^n (1+ \l/\gamma_j)\over \Pi_{i=1}^m (1+ \l/\beta_i)}=a_0+  \sum_{i=1}^m {a_i\over (\l+\beta_i)^{m_i+1}}$$

2) The distribution of the "over shoot"  $X_{T^x}-x$ on $T^x<\z$  ($T^x=\inf\{t; X_t>x\}$) is the following

$$\P(X_{T^x}-x\in dy; T^x<\z)= c_0(x)\delta_0(dy)+ \sum_{j=1}^n c_j(x) {y^{n_j} \over n_j!} e^{-\gamma_j y}1_{y>0}dy,$$

where the coefficients $c_j(x)$ are given by the rational expension :

$$c_0(x)+ \sum_{j=1}^n {c_j(x)\over (\l+\gamma_j)^{n_j+1}}= {\Pi_{i=1}^m (1+ \l/\beta_i)\over \Pi_{j=1}^n (1+ \l/\gamma_j)}.\sum_{i=1}^m{a_i\over m_i!} {\partial ^{m_i} \over (-\partial \beta)^{m_i}}\Bigl[{e^{-\beta x}\over \l+\beta}\Bigr]_{\beta=\beta_i}$$

\end{corollary}

Corollary \ref {Flu+} and Proposition \ref {Flu-}  give us a foretaste of the explicit results that we obtain in the sequel : Most of the explicit distributions that we will obtain will be of the preceding form or of convolution of measures of that form;  that is, a combination of 
exponential or gamma distributions with explicit coefficients, possibly  restricted to an interval,  and convoluted with (a possible restriction of) the measure $\cc U(dy)$.

 The next proposition introduces Laplace transforms of distributions that will be involved in the sequel. The easy proof is left to the reader. 
   
 \begin{proposition} \label{inf}
 
 1) $${\bf m}_i(x,\l):={\partial^{m_i} \over  (-\partial \beta)^{m_i}}\Bigl[{e^{(\l+\beta) x}\cc U_{[-x,0]}(-\beta)-\cc U_{[-x,0]}(\l) \over \l+\beta}\Bigr]_{\beta=\beta_i} $$ is the Laplace  transform of the measure $$\biggl((1_{y\in [0,x]}y^{m_i}e^{-\beta_i y}dy)*(1_{y\in [-x,0]}\cc U(dy))\biggr)1_{y\in [-x,0]}$$
 
2)  $${\bf n}_j(x,\l):= {\partial^{n_j} \over (-\partial \gamma)^{n_j}}\Bigl[{e^{-\l x}\cc U_{]-\i,-x[}(\l)- e^{\gamma x}\cc U_{]-\i,-x[}(-\gamma)\over \l+\gamma}\Bigr]_{\gamma=\gamma_j}$$ is the Laplace  transform of the measure 
$$\biggl( (y^{n_j} e^{-\gamma_j y}1_{y>0}) * (\cc U(dy-x)1_{y<0})\biggr)1_{y<0} $$

3) $${\bf  i}_j(\l)={\partial^{n_j}\over (-\partial \gamma)^{n_j}}\Bigl[{1\over \l+\gamma}\Bigr]_{\gamma=\gamma_j}$$ is the Laplace transform of the measure 
$$y^{n_j} e^{-\gamma_j y}1_{y>0}dy$$
 
\end{proposition}

\section{Some matrices of coefficients}

Recall first all of the notations already introduced.
$$-\gamma_1,\dots, -\gamma_n$$ are the poles of $\phi$ (and $\psi$) located on the half plane $\{\l, \Re(\l)<0\}$,
and $$n_j=\sharp\{k<j;\gamma_k=\gamma_j\}$$

$$-\beta_1,\dots, -\beta_m$$ are the roots of $\phi$ (and $\psi$) located on the half plane $\{\l, \Re(\l)<0\}$ together with $0$ is $\phi(0)=0$ and $\phi'(0)>0$. 
and $$m_i=\sharp\{k<i;\beta_k=\beta_i\}$$

We will denote by ${\bf l}$ and ${\bf c}$ the line and the column :

$${\bf l}=(1_{n_j=0}; 1\leq j\leq n).$$

$${\bf c}= (1_{m_i=0}; 1\leq i\leq m)^t.$$

The ${\bf m}_i(x,\l)$ and ${\bf n}_j(x,\l)$ and ${\bf  i}_j(\l)$ being defined in the previous proposition \ref{inf}, we denote  by  ${\bf m} (x,\l)$ the column  
$${\bf m} (x,\l)=({\bf m}_i(x,\l); 1\leq i\leq m)^t$$

and by ${\bf n} (x,\l)$ the line
$${\bf n} (x,\l)=({\bf n}_j(x,\l); 1\leq j\leq n)$$

and by ${\bf i}(\l)$ the line 
$${\bf i} (\l)=({\bf i}_j(\l); 1\leq j\leq n)$$

Also, define the line ${\bf v}(x)$, 

$${\bf v}(x):=\biggl(-\Bigl[{\partial^{n_j} \over(-\partial \gamma)^{n_j} } e^{\gamma x} \cc U_{]-\i,-x[}(-\gamma )\Bigr]_{\gamma=\gamma_j};  1\leq j\leq n\biggr), $$

and ${\bf w}(x)$ be the column,  

$${\bf w}(x):=\biggl(\Bigl[{\partial^{m_i} \over  (-\partial \beta)^{m_i}}e^{\beta_i x}\cc U_{[-x,0]}(-\beta)\Bigr]_{\beta=\beta_i};1\leq i\leq m\biggr)^t$$

\begin{proposition} \label{ccu}
If $m=n+1$, then the limit 
$$\cc u(-x):=\lim_{\l\to -\i} (-\l)e^{-\l x}\cc U_{]-\i,-x[}(\l)$$ exists in $[0,+\i[$ for all $x\in ]0,+\i[$.
\end{proposition}

This proposition will be proved incidently here as a consequence of our computations. In fact, it can be shown that  the function $\cc u$ is a density of the measure $1_{y<0}\cc U(dy)$ and we will see that this density is of bounded variations (see remark 4 after corollary  \ref {BccCco})

If $m=n+1$, we will denote ${\bf w}'(x)$ the column $\Bigl({\bf w}'_i(x); 1\leq i\leq m\Bigr)^t$, with 

$${\bf w}'_i(x):={\partial ^{m_i}\over (-\partial \beta)^{m_i}}\Bigl[ \beta e^{\beta x} \cc U_{[-x,0]}(\beta)+\cc u(-x)\Bigr]_{\beta=\beta_i}$$

Notice that if $\cc u$ is continuous, ${\bf w}'_i(x)$ is the left derivative of ${\bf w}_i(x)$, hence the notation.

Let now ${\bf W}(x)$ be the matrix $(m,n)$ : 

$${\bf W}(x)=({\bf W}_{i,j}(x); 1\leq i\leq m, 1\leq j\leq n)$$

where,

$${\bf W}_{i,j}(x):={\partial^{m_i}\partial^{n_j}\over(-\partial \beta)^{m_i} (-\partial \gamma)^{n_j}}\Bigl[{e^{\beta x}U_{[-x,0]}(-\beta)+ e^{\gamma x} \cc U_{]-\i,-x[}(-\gamma)\over \beta-\gamma}\Bigr]_{\beta=\beta_i,\gamma=\gamma_j}$$

When  $m=n+1$, denote $ {\bf \tilde W}(x)$ the square matrix 

$${\bf \tilde W} (x)=\left [\begin{array}{cc} {\bf w}(x) & {\bf W}(x)\end{array}\right]$$

We will  show later the next proposition

Let  ${\bf W}$ be any  matrix and let  ${\bf f}$ be any line, the dimension of  ${\bf f}$ being  equal to the number of columns of ${\bf W}$, ${\bf W}_i^{\bf f}$ will be  the matrix  obtained from $\bf W$ by replacing the  $i$-th line by $\bf f$. And ${\bf W}_i$ will denote  the matrix obtained from $\bf W$ by just taking it's $i$-th line off. 

Similarly, let ${\bf e}$ be a column whose  dimension  is equal to the number of lines of $W$,   ${\bf W}_{\bf e}^j$ will 
refer to   the matrix obtained from ${\bf W}$ by substituting ${\bf e}$ to  it's $j$-th column.

\section{The distribution of $(S_{\z}, X_{\z}, I_{\z})$}

\begin{theorem}\label{AccA}

The potential kernel  of the triple $(S_t, X_t, I_t)$ is characterized by the following  identities. 
For all $\l_1,\l_2 \in {\bf C}$,  $x\in ]0,+\i[$,  

$$\E( \int _0^{\z}e^{-\l_1 I_t} 1_{S_t-I_t\leq x} e^{-\l_2 (X_t-I_t)} dt)=\E( \int _0^{\z}e^{-\l_2 S_t}1_{S_t-I_t\leq x} e^{-\l_1 (X_t-S_t)} dt)$$
$$={\left |\begin{array} {cc} \cc U_{[-x,0]}(\l_1) & {\bf v} (x)  \\ {\bf  m}(x,\l_1)
& {\bf W}(x) \end{array} \right |\over |{\bf W}(x)|}{\left |\begin{array} {cc} 1 & {\bf l } \\ e^{-\l_2 x} {\bf  m}(x,\l_2)&  {\bf W}(x) 
\end{array} \right |\over |{\bf W}(x)|} \quad \hbox {if}\quad m=n$$

$$={ \left |\begin{array} {ccc} \cc U_{[-x,0]}(\l_1)  & \cc u(-x) & {\bf v}(x) \\ 
{\bf  m}(x,\l_1) &   {\bf w} (x) & {\bf W}(x) \end{array}\right |  \over | {\bf \tilde W}(x) |}{\left |\begin{array} {cc}e^{-\l_2  x} {\bf m}(x,\l_2) & {\bf W}(x)   \end{array} \right |\over |{\bf \tilde W}(x) | }\quad \hbox{if} \quad m=n+1$$

\end{theorem}

{\bf Notations}. Let,  if $m=n$, 
$$\cc a(x):=\left |\begin{array} {cc} \cc U_{[-x,0]}(0) & {\bf v} (x)  \\ {\bf  m}(x,0)
& {\bf W}(x) \end{array} \right |\qquad a(x):=\left |\begin{array} {cc} 1 & {\bf l} \\ {\bf  m}(x,0)&  {\bf W}(x) 
\end{array} \right | \qquad r(x):=|{\bf W}(x)|  $$

And, if $m=n+1$,

$$\cc a(x):= \left |\begin{array} {ccc} \cc U_{[-x,0]}(0)  & \cc u(-x) & {\bf v}(x) \\ 
{\bf  m}(x,0) &   {\bf w} (x) & {\bf W}(x) \end{array}\right |\qquad a(x):= \left |\begin{array} {cc} {\bf m}(x,0) & {\bf W}(x)   \end{array} \right |\qquad r(x):=|{\bf \tilde W}(x)|$$

\begin{corollary} \label{AccAcor}

If $\phi(0)>0$, then the law of the triple $(S_{\z}, X_{\z}, I_{\z})$ is characterized as follow :

 $${1\over \phi(0)}\P(S_{\zeta}-I_{\zeta}\leq x)={\cc a(x) a(x)\over r^2(x)}. $$
 
Given $S_{\zeta}-I_{\zeta}\leq x$, the random variable  $X_{\z}-I_{\z}$  is independent of $I_{\z}$ and has the same distribution as $S_{\z}$ and

 $$\P(I_{\z}\in dy\vert S_{\z}-I_{\z}\leq x)={1_{y\in [-x,0]}\over\cc a(x)}\biggl(\cc a_0 \cc U(dy)- [1_{y\in [-x,0]}\cc U(dy)]*[\sum_{i=1}^m \cc a_i(x) y^{m_i} e^{-\beta_i y}1_{y\in [0,x]}dy]\biggr)$$
 
$$\P(S_{\z}\in dy \vert S_{\z}-I_{\z}\leq x)={1_{y\in [0,x]}\over a(x)}\biggl(a_0\delta_0(dy)- [1_{y\in [0,x]}\cc U(dy-x)]*[ \sum_{i=1}^m a_i(x)y^{m_i} e^{-\beta_i y} 1_{y\in [0,x]}dy]\biggr) $$

with, if  $m=n$, 
$$\cc a_0=a_0=1$$
$$\cc a_i(x)={|{\bf W}_i^{{\bf v}(x)}(x) |\over |{\bf W}(x)|}\qquad a_i(x)={|{\bf W}^{\bf l}_i(x)|\over |{\bf W}(x)|},\qquad i=1\dots m$$

  And, if $m=n+1$,
 
$$\cc a_0=1\qquad a_0=0$$ 
 $$\cc a_i(x)= {|{\bf \tilde  W}(x) _i^{[ \cc u(-x), {\bf v} (x)]}(x) |\over  |{\bf\tilde W}(x)||}\qquad a_i(x)=-{ |{\bf W}_i(x)| \over |{\bf\tilde W}(x)|} \qquad i=1\dots m$$
 
 The identities are  still true  when $\phi(0)=0$ and when $\displaystyle  {1\over\phi(0)}.\P(S_{\z}-I_{\z}\leq x),$
 $$ \P(S_{\z}\in dy\vert S_{\z}-I_{\z}\leq x)\qquad \P(I_{\z}\in dy\vert S_{\z}-I_{\z}\leq x)$$ are replaced by $\displaystyle  \int_0^{+\i}\P(S_t-I_t\leq x)dt, $
 
 $${ \int_0^{+\i}\P(S_t-I_t\leq x, S_t\in dy))dt\over \int_0^{+\i}\P(S_t-I_t\leq x)dt}\qquad { \int_0^{+\i}\P(S_t-I_t\leq x, I_t\in dy)dt\over \int_0^{+\i}\P(S_t-I_t\leq x)dt}.$$

  \end{corollary}

  {\bf Remark}  Let $\phi_0$ be a proper (not killed) L\'evy exponent and apply the formula of theorem to $\phi=\phi_0+q$ when $q$ vary from $0$ to $+\i$, let $\P_q$ be the corresponding distribution, that is the distribution of the L\'evy process $X$, with L\'evy exponent $\phi_0$, killed at an independent exponential time of rate $q$. Clearly we have 
 $$\E_q( \int _0^{\z}e^{-\l_1 I_t} 1_{S_t-I_t\leq x} e^{-\l_2 (X_t-I_t)} dt)=\E_0(\int_0^{+\i}e^{-\l_1 I_t} 1_{S_t-I_t\leq x} e^{-\l_2 (X_t-I_t)}e^{- q t} dt) $$

  When taking the inverse Laplace transform over $q$, $\l_1$ and $\l_2$, 
  we find the measure 
 $$\P_0( I_t\in dy, X_t-I_t\in dz, S_t-I_t\leq x)dt$$
 Then the previous theorem and corollary give a characterization of the distribution of $(S_T,X_T,I_T)$ at any time $T$.
 
 As an application,  one can compute   the price of a double barrier knock-out call option under exponential L\'evy model for the asset $Y$, $Y_t=Y_0e^{X_t}$,  when the underlying L\'evy process $X$ satisfy the assumption \ref{Ass}, that is, an european call option with strike $K$  which  is desactivated if the asset $Y$ goes down a barrier value $a$ or up a value $b$ before the time of maturity $T$. It is clear that the  price is  given by the expression : 
 $$\E_0( (Y_0e^{X_T}-K) 1_{X_T>\log {K\over Y_0}} 1_{I_T>\log {a\over Y_0}}1_{S_T<\log {a\over Y_0}})$$
 
 Then, this price can be deduced from the distribution of the triple $(S_T, X_T, I_T)$.

\section{Behavior of the process $(X,S)$ at time $V_x=\inf\{t, X_t-S_t<-x\}$}
 
\begin{theorem}  \label {BccC}

 $$\E\biggl( \exp \bigl(-\mu_2 S_{V_x}-\mu_1(X_{V_x}-S_{V_x}))1_{V_x<\z}\bigr)\biggr) =\cc \psi(\mu_1){\left |\begin{array} {cc} 
\cc U_{]-\i, -x[}(\mu_1) & e^{\mu_1 x}.{\bf n} (x,\mu_1)\\ {\bf w}(x)  & {\bf W}(x) \end{array} \right |\over \left |\begin{array} {cc}  1 &  {\bf  i}(\mu_2)\\  {\bf w}(x) & {\bf W}(x)
\end{array} \right |}\quad \hbox{if}\quad m=n$$ and

$$=\cc \psi(\mu_1){\left |\begin{array}{ccc}\mu_1  \cc U_{]-\i, -x[}(\mu_1) &- \cc  U_{]-\i, -x[}(\mu_1) &  e^{\mu_1 x}.{\bf n}(x,\mu_1)\\
 {\bf  w'}(x)  & {\bf w}(x) & {\bf W}(x)\end{array}\right |\over \left |\begin{array}{ccc}\mu_2 & -1 &{\bf i}(\mu_2)\\
 {\bf  w'}(x)& {\bf w}(x) & {\bf W}(x)\end{array}\right |
}\quad \hbox{if}\quad m=n+1$$

\end{theorem}

{\bf Notations.} 
Denote if $m=n$,
$$b(x):=\left |\begin{array} {cc}  1 &  -{\bf  i}(0)\\  {\bf w}(x) & {\bf W}(x)
\end{array} \right |\qquad \cc c(x):=\lim_{\l \to 0}\cc  \psi(\l)\left |\begin{array} {cc} 
\cc U_{]-\i,-x[}(\l) & -{\bf n} (x,\l)\\ {\bf w}(x)  & {\bf W}(x) \end{array} \right |.$$

And, if $m=n+1$, 
$$b(x):=\left |\begin{array}{ccc}0 & -1 &{\bf i}(0)\\
 {\bf  w'}(x)& {\bf w}(x) & W(x)\end{array}\right |
\qquad \cc c(x):=\lim_{\l\to 0}\cc \psi(\l)\left |\begin{array}{ccc}\l  \cc U_{]-\i, -x[}(\l) & - \cc  U_{]-\i, -x[}(\l) & {\bf n}(x,\l)\\
 {\bf  w}'(x)  & {\bf w}(x) & W(x)\end{array}\right |$$

\begin{corollary} \label{BccCco} 

 $$\P(V_x<\z)={\cc c(x)\over b(x)},$$

The random variables $S_{V_x}$ and $X_{V_x}- S_{V_x}$ are independent conditionally to $V_x<\z$ and the distribution of $X_{V_x}-S_{V_x}$ on  $V_x<\z$ is characterized by the identity, 
$$\cc U(dz)* \P(X_{V_x}-S_{V_x}\in dz; V_x<\z)$$
$$=[b(x)]^{-1}.\Bigl[1_{z<-x}\cc U(dz)\Bigr] *\Bigl[\cc c_{-1} \delta' _0(dz) +\cc c_0(x)\delta_0(dz) -\sum_1^n  \cc c_j(x)z^{k_j}
 e^{-\gamma_j z}1_{z<-x}(dz)\Bigr]$$
And, 

 $$\P(S_{V_x}\in dy;  V_x<\z)=\cc c (x)\Bigl(b_0\delta_0(dy) + \sum_{i=1}^m  b_i(x)y^{m_i(x)} e^{-\beta_i(x) y}1_{y>0}dy\Bigr).$$

Where,  if $m=n$, 

$$\cc c_{-1}=0\qquad \cc c_0(x)=-1\qquad \cc c_j(x)=-\left | {\bf  W} (x)^j_{[{\bf w}(x)]}\right |  \qquad  j=0,\dots, n $$

And if $m=n+1$, then 

$$\cc c_{-1}=1\qquad \cc c_j(x)= \left |{\bf{\tilde W}} (x)^{j+1}_{[{\bf  w}'(x)]}\right | \qquad  j=0,\dots, n$$

The   $\beta_i(x)$, the exponents $m_i(x):=\sharp \{k<i; \beta_k(x)=\beta_i(x)\}$ and   the coefficients $b_i(x)$ are given by the following fractional expansion

$$b_0(x)+ \sum_{i=1}^m {b_i(x)\over (\l +\beta_i(x))^{m_i(x)+1}}= [\cc c_{-1}(x)\l +\cc c_0(x) - \sum_{j=1}^n  \cc c_j(x){1\over (\l+\gamma_j)^{n_j+1}}]^{-1}$$
$$= \left |\begin{array} {cc}  1 &  {-\bf  i}(\l)\\  1 & {\bf W}(x)\end{array} \right |^{-1} \quad \hbox{if}\quad  m=n$$

$$= \left |\begin{array}{ccc}\l &-1 & {\bf i}(\l)\\
 {\bf  w'}(x)& {\bf w}(x) & {\bf W}(x)\end{array}\right |^{-1}\quad \hbox{if}\quad m=n+1$$

\end{corollary}

{\bf Remark}

1) Notice that, when $\cc\psi(0)=0$, that means that $X$ does not drift to $+\i$ (see lemma 2, chapter 6 of \cite {B96}) then $$\P(X_{T_x}<\z)=\P(X_{T_x}<+\i)=1, $$  thus    $$\cc\psi(\l)\cc U_{]-\i,-x[}(\l)=\E( e^{-\l X_{T_x}}; T_x<+\i)\to 1\qquad  \l\to 0$$ and we get 
 $$\lim_{\l\to 0} \cc \psi(\l)n_j(x,\l)=\lim_{\l\to 0}{\partial^{n_j}\over(-\partial \gamma)^{n_j}} \Bigl[{\cc \psi(\l)e^{-\l x}\cc U_{]-\i,-x[}(\l)-\cc\psi(\l)e^{\gamma x}\cc U_{]-\i,-x[}(-\gamma)\over \l+\gamma}\Bigr]_{\gamma=\gamma_j}$$
 $$= {\partial^{n_j}\over(-\partial \gamma)^{n_j}} \Bigl[{1\over  \gamma}\Bigr]_{\gamma=\gamma_j}={ \bf i}_j(0)$$
 
 Thus,  we see  that  $b(x)=\cc c(x)$ when $\cc \psi(0)=0$ and,  according to the previous corollary, we obtain that $\P(V_x<+\i)=1$ when $X$ does not drift to $+\i$:  this means that the heights of the excursions out of the set $\{t, X_t=S_t\}$ are unbounded a.s.  This can be obtained by elementary trajectorial arguments.

2) In order to obtain a direct characterization of the distribution $\P(X_{V_x}-S_{V_x}\in dz; V_x<\z)$ instead of the characterization of the convoluted distribution 
$\cc U(dy)*\P(X_{V_x}-S_{V_x}\in dz; V_x<\z)$, it is necessary to introduce the drift, the L\'evy measure and the killing rate associated to the exponent   $\cc\psi$. We leave this part to the interested reader.

3) Due to  the lake of memory of exponential distributions and the fact that 
$$X_{V_x}1_{X_{V_x}>0}= (S_{V_x}+ (X_{V_x}-S_{V_x}) )^+, $$
 the distribution of $X_{V_x}$ given $X_{V_x}>0$ is, like the distribution of $S_{V_x}$ as stated  in the preceding corollary,  a linear combination of the functions $y^{m_i(x)}e^{-\beta_i(x)y}1_{y>0}$.
  
4) It is clear  from the characterisation of the distribution of $X_{V_x}-S_{V_x}$ when $m=n+1$ that the convolution $1_{z<-x}\cc U(dz)*\delta'(dz)$ is a signed measure, that means that the measure $\cc U(dy)$ has a density, and this density, which is the fonction $\cc u$ appearing in lemma \ref{ccu},  is of bounded variations.

\section{Behavior of the process $(I,X)$ at time $V^x:=\inf\{t; X_t-I_t>x\}$}

\begin{theorem}  \label{CccB}
$$\E( e^{-\mu_1 I_{V^x}} e^{-\mu_2 (X_{V^x}-I_{V^x})})={e^{-\mu_2 x}\over \cc \psi(\mu_1)}{\left |\begin{array} {cc} 0 & {\bf  i}(\mu_2)  \\   {\bf c} & W(x) \end{array} \right | \over 
\left |\begin{array} {cc} 1& {\bf  n}(x,\mu_1)
 \\  {\bf c} &  W(x)  \end{array} \right |}\quad\hbox{if}\quad m=n $$
 
$$= {e^{-\mu_2 x}\over\cc \psi (\mu_1)}.{\left |\begin{array} {ccc} 0 & 1 &{\bf  i}(\mu_2) \\    {\bf c} & {\bf w} (x) &{\bf  W} (x)   
\end{array} \right | \over \left |\begin{array} {ccc} 1 & e^{-\mu_1 x} \cc U_{]-\i,-x[}(\mu_1) &{\bf  n}(x,\mu_1) \\   {\bf c} & {\bf w}(x) &{\bf W}(x)   
\end{array} \right |} \quad\hbox{if}\quad m=n+1 $$

\end{theorem}

{\bf Notations} If $m=n$,
$$c(x)=\left |\begin{array} {cc} 0& {\bf  i}(0)
 \\   {\bf c} &  {\bf W}(x)  \end{array} \right |\qquad \cc b(x)=\lim_{\l\to 0}\cc \psi(\l)\left |\begin{array} {cc} 1& {\bf  n}(x,\l)
 \\  {\bf c} &  {\bf W}(x)  \end{array} \right |$$
 
 If $m=n+1$, 
 
 $$ c(x)= \left |\begin{array} {ccc} 0 &  1 &{\bf  i}(0) \\    {\bf c} &  {\bf w} (x) & {\bf W}(x)  \end{array} \right |\qquad \cc b(x)=\lim_{\l\to 0}\cc \psi(\l)\left |\begin{array} {ccc} 1 & e^{-\l x} \cc U_{]-\i,-x[}(-\l ) &{\bf  n}(x,\l) \\   {\bf c} & {\bf w}(x) & {\bf W}(x)   
\end{array} \right |$$

One can notice again that $\cc b(x)=c(x)$ when $\cc\psi(0)=0$.

\begin{corollary} \label {CccBco}

$$\P(V^x<\z)= {c(x)\over \cc  b(x)}$$

and $I_{V^x}$ and $X_{V^x}-I_{V^x}$ are independent conditionally to $V^x<\z$

$$\P(X_{V^x}-I_{V^x}\in x+ dy\vert  V^x<\z)=[c(x)]^{-1}.\Bigl(c_0(x)\delta_0(dy) + \sum_{j=1}^n c_j(x)y^{n_j} e^{-\gamma_j y}1_{y>0} dy\Bigr) $$

$$\P(I_{V^x}\in dz; V_x<\z)={c (x) \over r(x)} \Bigl[\cc U(dz)\Bigr]* \Bigl[\sum_{n=0}^{+\i}\mu_x^{*n}(dz)\Bigr]$$

with $$\mu_x(dz)=\bigl[1_{z<0}.\cc U(dz-x)\bigr]*\bigl [c_0(x)\delta_0(dz) + \sum_{j=1}^n  c_j(x)  z^{n_j} e^{-\gamma_j z}1_{z>0} dz\bigr]1_{z<0},$$

and, if $m=n$, 
$$c_0(x)=0\quad \hbox{ and  }\quad c_j(x)={|W^j_{\bf c}(x)|\over |{\bf W}(x)|} \quad  \hbox{for}\quad j=1,\dots, n $$

If $m=n+1$, 
$$c_j(x)={|{\bf  \tilde W }^{j+1}_{\bf c}(x)|\over |{\bf \tilde W}(x)|} \quad \hbox{for}\quad  j=0,\dots, n $$

\end{corollary}

\section{Behavior of the process  at time $U_x=\inf\{t, S_t-I_t>x\}$.}
 
 \begin{theorem} \label{Ux}
 $$\P(U_x<\z, X_{U_x}=I_{U_x})={a(x)\cc c(x)\over r(x)^2}\qquad \P(U_x<\z, X_{U_x}=S_{U_x})={\cc a(x)
 c(x)\over r(x)^2}$$

Given  the event $\{U_x<\z;X_{U_x}=I_{U_x}\}$,    $S_{U_x}$ and  $X_{U_x}-S_{U_x}$ are independent.  The random variable   $S_{U_x}$ has  the same distribution as  $S_{\z}$ given $S_{\z}-I_{\z}\leq x$ if  $\phi(0)>0$, and has the distribution $\displaystyle {\int_0^{\i}\P(S_t\in dz,S_t-I_t\leq x) dt \over \int_0^{\i}\P(S_t-I_t\leq x) dt} $ if $\phi(0)=0$. The random variable $X_{U_x}-S_{U_x}$ 
has the same distribution as $X_{V_x}-S_{V_x}$ given $V_x<\z$.

Given the  event $\{U_x<\z;X_{U_x}=S_{U_x}\}$,   $I_{U_x}$ and  $X_{U_x}-I_{U_x}$ are independent.  The random variable $I_{U_x}$ has  the same distribution as  $I_{\z}$ given $S_{\z}-I_{\z}\leq x$ if $\phi(0)>0$, and 
has the distribution $\displaystyle {\int_0^{\i}\P(I_t\in dy,S_t-I_t\leq x)  dt \over \int_0^{\i}\P(S_t-I_t\leq x) dt } $ if $\phi(0)=0$. 
The random variable $X_{U_x}-I_{U_x}$ has the same distribution as $X_{V^x}-I_{V^x}$ given $V^x<\z$.

\end{theorem}

\section{Behavior of the process at time $T_a^b=\inf\{t; X_t\not\in [-a,b]\}$}

\begin{theorem} \label {Tab} We have, if 
$\phi(0)>0$

 $$\P(X_{T_a^b}=S_{T_a^b}, T_a^b<\z)=\int_{[-a,0]} {c( b-y)\over a(b-y)}{\P(I_{\z}\in dy, S_{\z}\leq b) \over \phi(0)}.$$

And 

 $$\P(X_{T_a^b}=I_{T_a^b}, T_a^b<\z)=\int_{[0,b]} {\cc c( a+x)\over \cc a(a+x)}{\P(I_{\z}\geq -a, S_{\z}\in dx)  \over \phi(0)}$$
 
 If $\phi(0)=0$, we   replace in last formulas ${\P(I_{\z}\in dy, S_{\z}\leq b) \over \phi(0)} $ and ${\P(I_{\z}\geq -a, S_{\z}\in dx ) \over \phi(0)}$ by respectively 
 
$$\int_0^{+\i}\P(I_t\in dy,S_t\leq b) dt \quad \hbox{and} \quad \int_0^{+\i}\P(S_t\in dx,-I_t\geq -a) dt $$

Given $\{X_{T_a^b}=S_{T_a^b}, T_a^b<\z\}$,  the distribution of $I_{T_a^b}$ is the distribution of $I_{\z}$ given $\{S_{\z}\leq b, I_{\z}\geq - a\}$ if $\phi(0)>0$ and the distribution  $\displaystyle {(\int_0^{+\i}\P(I_t\in dy,S_t\leq b) dt \over \int_0^{+\i}\P(S_t\leq b,-I_t\geq -a) dt}$ if $\phi(0)=0$. The distribution of $X_{T_a^b}-I_{T_a^b}$ given $I_{T_a^b}=y$ is the one of $X_{V^{b+y}}$ given $V^{b+y}<\z$.

Given $\{X_{T_a^b}=I_{T_a^b}, T_a^b<\z\}$, the distribution of $S_{T_a^b}$ is the distribution of $S_{\z}$ given $\{S_{\z}\geq -a, I_{\z}\leq b\}$ if $\phi(0)>0$ and the distribution  $\displaystyle {\int_0^{+\i} \P(S_t\in dx, -I_t\geq -a) dt \over \int_0^{+\i}\P(S_t\leq b;  -I_t\geq -a) dt}$ if $\phi(0)=0$. 
The distribution of $X_{T_a^b}-S_{T_a^b}$ given $S_{T_a^b}=x$ is the one of $X_{V_{a+x}}$ given $V_{a+x}<\z$. 

\end{theorem}

{\bf Remarks} 1) We can deduce from this theorem that the distribution of $X_{T_a^b}-b$ given $X_{T_a^b}>b$ is a linear combination of distribution of the distributions $y^{n_j} e^{-\gamma_j y} 1_{y>0}dy $. That could easily be deduced from  the lack of memory of the exponential distribution. The previous  result allows  also to compute a closed form of the coefficients.

2) As in the remark following theorem \ref{AccA}, let $\phi_0$ be a proper (not killed) L\'evy exponent and apply the formula of theorem to $\phi=\phi_0+q$ when $q$ vary from $0$ to $+\i$, let $\P=\P_q$ be the corresponding distribution, that is the distribution of the L\'evy process $X$, with L\'evy exponent $\phi_0$, killed at an independent exponential time of rate $q$.
The probabilities $\P_q(X_{T_a^b}=S_{T_a^b}, T_a^b<\z)$  and $\P_q(X_{T_a^b}=I_{T_a^b}, T_a^b<\z)$ are $q$ times the Laplace transforms  of the functions of $t$, 

$$\P_0(X_{T_a^b}=S_{T_a^b}, T_a^b<t)\quad\hbox{and}\quad \P_0(X_{T_a^b}=S_{T_a^b}, T_a^b<t).$$ 

On the other hand, 
the price of a double barrier with rebate, say rebate $r^+$  (resp. $r^-$)  if the option exceed the value $b$  (resp. goes down the value $a$) before time $T$. (Denote $\tilde a=-\log {a\over Y_0}$, $\tilde b=\log {b\over Y_0}$) is given by the next formula

$$\E( (Y_0e^{X_T}-K) 1_{X_T>\log {K\over Y_0} }1_{I_T>-\tilde a; S_T<\tilde b})+ r^+\P( X_{T^{\tilde b}_{\tilde a}}= S_{T^{\tilde b}_{\tilde a}}, T^{\tilde b}_{\tilde a} <T)$$
 
$$+r^-\P( X_{T^{\tilde b}_{\tilde a}} = I_{T^{\tilde b}_{\tilde a}}, T^{\tilde b}_{\tilde a}<T)$$

We have seen that the first term can be computed with the help of  theorem \ref{AccA} and now we see that the two other terms can be computed  by  an  inverse  Laplace transform in time of 
$\P_q( X_{T^{\tilde b}_{\tilde a}}= S_{T^{\tilde b}_{\tilde a}}, T^{\tilde b}_{\tilde a} <\z)$  and $\P_q( X_{T^{\tilde b}_{\tilde a}} = I_{T^{\tilde b}_{\tilde a}}, T^{\tilde b}_{\tilde a}<\z)$  
which are characterized in the preceding theorem.

 \section{Recall of the main result  of \cite{F10}  and its application with the Assumption \ref{Ass} } 
 
  \begin {proposition} \label{funct}
 There exist unique functions $A(x,\l)$, $\cc A(x,\l)$, $B(x,\l)$ and $\cc B(x,\l)$, $C(x,\l)$  and $\cc C(x,\l)$ such that 
  $A(x,\l)$ and $\cc A(x,\l)$,
  are  defined for $\l\in \C$ and entire;   $B(x,\l)$ and $C(x,\l)$ are defined, continuous  on $\{\Re(\l)\geq 0\}$ and holomorphic on $\{\Re(\l)>0\}$; $\cc B(x,\l)$ and $\cc C(x,\l)$ are defined and  continuous on  $\{\Re(\l)\leq 0\}$, and holomorphic on $\{\Re(\l)<0\}$, satisfying the following identities,
 
$$\P( \int _0^{\z}e^{-\l_1 I_t} 1_{S_t-I_t\leq x} e^{-\l_2 (X_t-I_t)} dt)=\P( \int _0^{\z}e^{-\l_2 S_t}1_{S_t-I_t\leq x} e^{-\l_1 (X_t-S_t)} dt)$$
 $$=\cc A(x,\l_1)A(x,\l_2) $$

$$\E\biggl( \exp \bigl(-\mu_2 S_{V_x}-\mu_1(X_{V_x}-S_{V_x})1_{V_x<\z}\bigr)\biggr) ={\cc C(x,\mu_1)
\over B(x,\mu_2)} $$

$$\E\biggl( \exp \bigl(-\mu_1 I_{V^x}-\mu_2(X_{V^x}-I_{V^x})1_{V^x<\z}\bigr)\biggr) ={C(x,\mu_2)\over \cc B(x,\mu_1)}  $$

$$\psi(\l)A(x,\l)\to 1 \qquad B(x,\l)\sim \psi(\l) \quad \hbox{for}\quad \l\to +\i$$
$$\cc B(x,\l)\sim\cc\psi(\l)\quad \hbox{for}\quad  \l\to-\i$$
 
 \end {proposition}

 \begin{proof}
 
 The existence of the  6 functions $A$,$\cc A$, $B$, $\cc B$, $C$, $\cc C$ has been settled in the  general setting of a L\'evy process in proposition 6.1 and proposition 6.2 of \cite{F10} , the behavior of the functions  is given in theorem 4.2 of \cite{F10}.  The uniqueness is obvious.

\end{proof} 
More over, we know  from theorem 4.2 of  \cite{F10} (or we can deduce them from  proposition  \ref{funct}),   the behavior of the six functions as settled in next proposition

\begin{proposition}\label{behav}

The functions of $\l$, $A(x,\l)$,  $e^{-\l x} \cc A(x,\l)$ , ${\cc B(x,\l)\over |\l|+1}$, $e^{\l x}C(x,\l)$ are bounded on $\{Re(\l)\geq 0\}$.  

The functions $\cc A(x,\l)$,  $e^{\l x} A(x,\l)$, ${\cc B(x,\l)\over |\l|+1}$, $e^{-\l x}\cc C(x,\l)$ are bounded on $\{Re(\l)\leq 0\}$. 

For $\Re(\l)\to +\i$  $e^{\l x}C(x,\l)\to 0$ if $\psi(\l)$ is bounded, and  for $\Re(\l)\to -\i$,  $e^{-\l x}\cc C(x,\l)\to 0$   if $\cc\psi(\l)$ is bounded.

\end{proposition}
{\bf Remark} The property "$\psi(\l)$ is bounded" means that $\psi$ is the exponent of a compound  Poisson process and 
this is equivalent to the property  (see chapter 6 of \cite{B96} that the first time the L\'evy process  $X$  visits  $]0,+\i[$ is non zero a.s. (here it is also equivalent to the condition $m=n$).

In the next lemma, we state that the determination  of these six functions reduces to  the computation of four polynomials.

 \begin{proposition}  \label{poly}
 
 There exist unique polynomials of $\l$, $P_1(x,\l)$, $Q_1(x,\l)$, $P_2(x,\l)$, $Q_2(x,\l)$, with  
 $$\deg P_1=n\quad \deg P_2=m\quad \deg Q_1\leq m-1\quad \deg Q_2\leq m-1,
 $$ such that 
 
 $$A(x,\l) ={P_1(x,\l)+ Q_1(x,\l)e^{-\l x}\cc U_{[-x,0]}(\l)\over \Pi_1^m(\l+\beta_i)}$$
 $$ \cc A(x,\l)={Q_2(x,\l)e^{\l x}+ P_2(x,\l)\cc U_{[-x,0]}(\l)\over \Pi_1^m(\l+\beta_i)}$$
 $$B(x,\l)={P_2(x,\l)\over  \Pi_1^n(\l+\gamma_j)}$$
 $$\cc B(x,\l)=\cc\psi(\l) {(P_1(x,\l)- Q_1(x,\l)e^{-\l x} \cc U_{]-\i, -x[}(\l))\over \Pi_1^n (\l+\gamma_j)}$$
 $$C(x,\l)={e^{-\l x}Q_1(x,\l)\over  \Pi_1^n (\l+\gamma_j)}$$
 $$ \cc C(x,\l)=\cc\psi(\l).{(P_2(x,\l) \cc U_{]-\i, -x[}(\l)- Q_2(x,\l)e^{\l x}) \over \Pi_1^n (\l+\gamma_j)}$$
 
\end{proposition}

\begin{proof} {\it  of proposition \ref {poly} , proposition \ref{ccu},  }
We have seen in \cite{F10} theorem 4.2 that the matrix $M(x,\l)$ defined as follows
$$M(x,\l):=\left( \begin{array} {clcr} A(x,\l) & -C(x,\l) \\ \cc A(x,\l)  & B(x,\l)\end{array} \right)\hbox{ if } \quad \Re (\l)>0,$$

$$M(x,\l):=\left( \begin{array} {clcr} \cc B(x,\l) & A(x,\l) \\ -\cc C(x,\l)  & \cc A(x,\l)\end{array} \right)\hbox{ if } \quad \Re (\l)<0.$$
satisfies the identity 

\begin{equation} \label {WHM} M^+(x,iu)=M^-(x,iu) \left( \begin{array} {clcr} 0 & -1\\ 1& \phi(iu) \end{array} \right)\end{equation}

Moreover, $\det M(x,\l)=1$, and $M$ is invertible.
On the other hand, one can easily check that the matrix 
$$N(x,\l)=\left( \begin{array} {clcr}  1 &0 \\  \cc U_{[-x,0]}(\l)  & 1\end{array} \right)\left( \begin{array} {clcr}  \Pi_1^m(\l+\gamma_j)  & 0 \\ 0 & \Pi_1^m(\l+\beta_i)\end{array} \right)
\hbox{for} \quad \Re (\l)>0$$

$$N(x,\l)=\left( \begin{array} {clcr}  \cc \psi(\l) & 1 \\ - \psi(\l).\cc U_{]-\i,-x[}(\l)& \cc U_{[-x,0]}(\l)   \end{array} \right)\left( \begin{array} {clcr}  \Pi_1^m(\l+\beta_i)  &0 \\ 0 & \Pi_1^n(\l+\gamma_j)\end{array} \right)
\hbox{for} \quad \Re (\l)<0$$

satisfies the same  identity \ref{WHM}.

Thus, the product $N(x,\l)M^{-1}(x,\l)$ is entire. One can see easily that, because of the boundary conditions fulfilled by the components of $M(x,\l)$ given in
proposition \ref {behav},  that this matrix $N(x,\l)M^{-1}(x,\l)$ is of the form :

$$N(x,\l)M^{-1}(x,\l)=\left( \begin{array} {clcr} P_2(x,\l) & -e^{-\l x}Q_1(x,\l)\\ 
-e^{\l x} Q_2(x,\l)&P_1(x,\l)\end{array} \right)=:R(x,\l)$$

Where $P_1$, $Q_1$, $P_2$ and $Q_2$ are polynomials of $\l$.
Then $$\det R=\det N=\Pi_1^m(\l+\beta_i)\Pi_1^n(\l+\gamma_j), $$ and 
we  obtain 
$$M(x,\l)=R^{-1}(x,\l).N(x,\l)$$
$$=\left( \begin{array} {clcr} P_1(x,\l) & e^{-\l x}Q_1(x,\l)\\ 
e^{\l x}Q_2(x,\l) &P_2(x,\l)\end{array} \right)\times 
\left(\begin{array} {clcr}  1 &0 \\  \cc U_{[-x,0]}(\l)  & 1\end{array} \right)$$
$$\times \left( \begin{array} {clcr}  {1\over \Pi_1^m(\l+\beta_i)}  &0 \\ 0 & {1\over \Pi_1^n(\l+\gamma_j)}\end{array} \right)\hbox{for} \quad \Re (\l)>0$$

$$=\left( \begin{array} {clcr} P_1(x,\l) & e^{-\l x}Q_1(x,\l)\\ 
e^{\l x} Q_2(x,\l) &P_2(x,\l)\end{array} \right)\times
\left( \begin{array} {clcr}  \cc \psi(\l) & 1 \\  -\cc\psi(\l).\cc U_{]-\i,-x[}(\l)&\cc U_{[-x,0]}(\l)   \end{array} \right)$$
$$\times \left( \begin{array} {clcr}  {1\over \Pi_1^n(\l+\gamma_j)}  &0 \\ 0 & {1\over \Pi_1^m(\l+\beta_i)}\end{array} \right)\hbox{for} \quad \Re (\l)<0$$

Developping each term of this matrix product, we obtain the identities of the theorem. Now, according to proposition \ref {behav}, $e^{-\l x} C(x,\l)$ is bounded on $\{\Re(\l)>0\}$ and goes to $0$ if $\psi$ is bounded (that is if $m=n$), that gives that $\deg Q_1\leq m-1$. Futhermore, the equivalence of the proposition \ref{funct} $B(x,\l)\sim \psi(\l)$ gives  that $P_2(x,\l)=\l^m$ for $\l\to +\i$. 
Since  $\psi(\l)A(x,\l)\to 1$, we deduce that $P_1(\l)\sim \l^n$ for $\l\to +\i$.

Also $e^{-\l x} \cc C(x,\l)$ is bounded for $\Re(\l)\to +\i$ and goes to $0$ if $\cc\psi(\l)$ is bounded. This is equivalent to say that 

$${P_2(x,\l) e^{-\l x}\cc U_{]-\i,-x[}(\l)-Q_2(x,\l)\over \Pi_{j=1}^n(\l+\gamma_j)}\to 0\quad\hbox{if}\quad  \Re(\l)\to -\i$$

In other words, \begin{equation} P_2(x,\l) e^{-\l x}\cc U_{]-\i,-x[}(\l)-Q_2(x,\l)=o(\l^n)\label{ln} \end{equation}

If $m=n$ then $P_2(x,\l)\sim \l^n$ and we have clearly that  $e^{-\l x}\cc U_{]-\i,-x[}(\l)\to 0$ These two facts put together with  property (\ref{ln}) give us that $Q_2(x,\l)=o(\l^n)$ and so, 

$$\deg Q_2(x,\l)\leq n-1=m-1$$

If $m=n+1$, then $P_2(x,\l)\sim \l^{n+1}$ and   $P_2(x,\l) e^{-\l x}\cc U_{]-\i,-x[}(\l)=o(\l^{n+1})$ then, with property (\ref{ln}),  we deduce that $Q_2(x,\l)=o(\l^{n+1})$, in other words,   

$$\deg Q_2(x,\l)\leq n=m-1$$

The uniqueness of the 4 polynomials come from the uniqueness of the 6 functions. This finish the proof of proposition \ref{poly}. 

Now, when $m=n+1$ and when looking more precisely the preceding identities, since $\deg Q_2(x,\l)\leq n$, we have that ${Q_2(x,\l)\over \Pi_{j=1}^n(\l+\gamma_j)}$ has a finite limit. Since 
$${P_2(x,\l) e^{-\l x}\cc U_{]-\i,-x[}(\l)-Q_2(x,\l)\over \Pi_{j=1}^n(\l+\gamma_j)}\to 0 \qquad \l \to -\i $$ 
we deduce that 

$${P_2(x,\l) e^{-\l x}\cc U_{]-\i,-x[}(\l)\over \Pi_{j=1}^n(\l+\gamma_j)}\hbox{ has a finite limit when } \l \to -\i$$

Since $P_2(x,\l) \sim \l^{n+1}$ and $\Pi_{j=1}^n(\l+\gamma_j)\sim \l^n$,  we obtain that 

$\l e^{-\l x}\cc U_{]-\i,-x[}(\l)$ has a finite limit ($-\cc u(-x)$)  when $\l \to -\i$.  That proves proposition \ref{ccu}.  Now since ${P_2(x,\l) e^{-\l x}\cc U_{]-\i,-x[}(\l)-Q_2(x,\l)\over \Pi_{j=1}^n(\l+\gamma_j)}$  goes to $0$( for  we deduce that ${Q_2(x,\l)\over \l^n}\to -\cc u(-x)$. We state this property in the next lemma. \end{proof}

\begin{lemma} \label{q2} $$\lim_{|\l|\to+\i}{Q_2(x,\l)\over \l^n}= -\cc u(-x)$$

\end{lemma}

\section {Computation of $A$ and $\cc A$ and proof of theorem  \ref{AccA} and corollary \ref{AccAcor} }

\begin{proposition}\label{calA}
$$A(x,\l)={\left |\begin{array} {cc} 1 & {\bf l } \\  e^{-\l x}{\bf  m}(x,\l)&  {\bf W}(x) 
\end{array} \right |\over |W(x)|} \quad \hbox { when $m=n$,}$$

$$A(x,\l) ={\left |\begin{array} {cc} e^{-\l x}{\bf m}(x,\l) & {\bf W}(x)   \end{array} \right |\over |{\bf \tilde W}(x)| }\quad \hbox{ when } m=n+1 
$$
\end{proposition}   

\begin{proof} 

Recall   the equation of proposition \ref{funct} 
$$A(x,\l) ={1\over \Pi_1^m(\l+\beta_i)}\biggl(P_1(x,\l)+ Q_1(x,\l)e^{-\l x}\cc U_{[-x,0]}(\l)\biggr)$$
 Since $A(x,\l)$ is an entire function and because $\deg P_1=n$  and $\deg Q_1\leq m-1$  (see proposition \ref {poly}), then $A(x,\l)$ is necessarily of the next  form 
  
  $$A(x,\l)=a_0(x)-\sum_1^m a_i(x){\partial^{m_i} \over (-\partial \beta)^{m_i}}\Bigl[{e^{\beta x}\cc U_{[-x,0]}(-\beta)-e^{-\l x}\cc U_{[-x,0]}(\l)\over \beta+\l }\Bigr]_{\beta=\beta_i}$$
  $$=a_0(x)-\sum_1^m a_i(x)e^{-\l x}{\bf m}_i (x,\l)$$
 
for some coefficients $a_i(x)$, $i=0,1,\dots, m$ 
 
Thus, 
$$P_1(x,\l)=\Pi_1^m(\l+\beta_i)\biggl( a_0(x)-\sum_1^m a_i(x){\partial^{m_i} \over (-\partial \beta)^{m_i}} \Bigl[{e^{\beta x}\cc U_{[-x,0]}(-\beta)\over \l+\beta}\Bigr]_{\beta=\beta_i}\biggr)$$
 
$$Q_1(x,\l)=\Pi_1^m(\l+\beta_i)\biggl(\sum_1^m a_i(x){\partial^{m_i} \over (-\partial \beta)^{m_i}}\Bigl[{1\over \l+\beta}\Bigr]_{\beta=\beta_i}\biggr)$$

Now using  property  $P_1(x,\l)\sim \l^n$ of lemma  \ref {poly}, one obtain 

\begin{equation} \label{a0n} a_0(x)=1 \quad \hbox{ if }\quad m=n \end{equation}

\begin{equation}\label{a0m} a_0(x)=0  \quad \hbox{ and }\quad \sum_{i=1}^m a_i(x){\bf w}_i(x)=-1  \quad \hbox{ if }\quad m=n +1 \end{equation}

On the other  hand, take the identity 
$$\cc B(x,\l)={\cc \psi(\l)\over \Pi(\l+\gamma_j)}.\biggl(P_1(x,\l) - Q_1(x,\l)e^{-\l x}\cc U_{]-\i,-x[}(\l)\biggr),$$
and replace $P_1$ and $Q_1$ by the above expressions, we get 
$$\cc B(x,\l)=\cc \psi(\l){\Pi_1^m(\l+\beta_i)\over \Pi_1^n(\l+\gamma_j)}\biggl(a_0(x)-\sum_{i=1}^m a_i(x){\partial^{m_i} \over (-\partial \beta)^{m_i}}\Bigl[ {e^{\beta x}\cc U_{[-x,0]}(-\beta)+e^{-\l x}\cc U_{]-\i,-x[}(\l)\over \beta+\l})\Bigr]_{\beta=\beta_i}\biggr)$$

The function $\cc B(x,\l)$ is   holomorphic on  the left half plan,  in particular at points $\l=-\gamma_j$, this implies that 
for every $j=1,\dots, n$, 

$${\partial^{n_j} \over (\partial \lambda)^{n_j}}\sum_{i=1}^m a_i (x) {\partial^{m_i} \over (-\partial \beta)^{m_i} }\Bigl[{e^{\beta x}\cc U_{[-x,0]}(-\beta)+e^{-\l x}\cc U_{]-\i,-x[}(\l)\over \beta+\l})\Bigr]_{\beta=\beta_i,\lambda=-\gamma_j}={\partial^{n_j} \over (\partial \l)^{n_j}}[a_0(x)]_{\l=-\gamma_j}$$

In other words, $$\sum_{i=1}^m a_i (x) W_{i,j}(x)=1_{n_j=0} \qquad \hbox{ if }\quad m=n$$
And $$\sum_{i=1}^m a_i (x) W_{i,j}(x)=0 \qquad \hbox{ if } m=n+1 $$
These $n$ equations added to equation (\ref{a0m}) if $m=n+1$,   form a linear system of $m$ equations with  variables $a_i(x)$, $i=1,\dots, m$.  It is easy to be convinced that 
each solution of this systems brings a new couple of polynomials $P_1(x,\l)$ and $Q_1(x,\l)$. The uniquiness of such a couple leads to
the uniquiness of this solution. Thus the system is a Cramer System and the determinant of  the matrix ${\bf W}(x)$    for $m=n$ (resp. of the matrix  ${\tilde W} (x)$  for $m=n+1$) does not vanish. Finally we obtain, 

For $m=n$,

$$a_0(x)=1\qquad a_i(x)={|{\bf W}_i^{\bf l}(x)|\over |{\bf W}(x)|} \quad  \hbox{for}\quad i= 1,\dots, n $$
and  
$$A(x,\l) =a_0(x)-\sum_1^m a_i(x)e^{-\l x} {\bf m}_i(x,\l)={\left |\begin{array} {cc} 1 & {\bf 1} \\ e^{-\l x} {\bf  m} (x,\l)& {\bf  W}(x)
\end{array} \right |\over |{\bf W}(x)|}.$$

For $m=n+1$, $$a_0(x)=0\qquad a_i(x)=-{|{\bf W}_i(x)|\over |{\bf \tilde W} (x)|}   \quad  \hbox{for}\quad i= 1,\dots, m, $$ 
and 
$$A(x,\l) =-\sum_1^m a_i(x)e^{-\l x} {\bf m}_i(x,\l)={\left |\begin{array} {cc} e^{-\l x} {\bf  m} (x,\l)& {\bf W}(x)
\end{array} \right |\over  |{\bf \tilde W} (x)|}$$

\end{proof}

\begin{proposition}\label{calcA}
$$\cc A(x,\l)={\left |\begin{array} {cc} \cc U_{[-x,0]}(\l_1) & {\bf v} (x)  \\ {\bf  m}(x,\l_1)
& {\bf W}(x) \end{array} \right |\over |{\bf W}(x)|} \quad \hbox{if} \quad m=n $$

$$\cc A(x,\l) ={ \left |\begin{array} {ccc} \cc U_{[-x,0]}(\l_1)  & \cc u(-x) & {\bf v}(x) \\ 
{\bf  m}(x,\l_1) &   {\bf w} (x) & {\bf W}(x) \end{array}\right |  \over |{\bf \tilde W}(x)|}\quad \hbox{if} \quad m=n+1$$
\end{proposition}

\begin{proof} Take the equation of proposition \ref{funct},
$$\cc A(x,\l)={1\over \Pi_1^m(\l+\beta_i)}\Bigl(Q_2(x,\l)e^{\l  x} + P_2(x,\l)\cc U_{[-x,0]}(\l)\Bigr). $$
 Since 
$\cc A(x,\l)$ is an entire function of $\l$ and $P_2(\l)\sim\l^m$, $\deg Q_2\leq m-1$ according to proposition \ref{poly}, then $\cc A(x,\l)$ is necessarily of the next form 
$$\cc A(x,\l)=\cc U_{[-x,0]}(\l)+\sum_1^m \cc a_i(x){\partial ^{m_i}\over (-\partial \beta)^{m_i}} \Bigl[{\cc  U_{[-x,0]}(\l)- e^{(\l +\beta)x}\cc U_{[-x,0]}(-\beta)\over \l+\beta}\Bigr]_{\beta=\beta_i}$$
$$=\cc U_{[-x,0]}(\l) - \sum_1^m \cc a_i(x){\bf m}_i(x,\l) $$

For some coefficients $\cc a_i(x)$, $i=1,\dots, m$. Thus, 
 
 $$Q_2(x,\l)=- \Pi_1^m(\l+\beta_i)\Bigl( \sum_1^m \cc a_i(x){\partial ^{m_i}\over (-\partial \beta)^{m_i}} \Bigl[{e^{\beta x}\cc U_{[-x,0]}(-\beta_i)\over \l+\beta}\Bigr]_{\beta=\beta_i}\Bigr)$$
 
$$P_2(x,\l)= \Pi_1^m(\l+\beta_i)\Bigl(1+\sum_1^m\cc a_i(x){\partial ^{m_i}\over (-\partial \beta)^{m_i}} \Bigl[{1\over \l+\beta}\Bigr]_{\beta=\beta_i}\Bigr)$$

Take now the identity  of proposition \ref{funct},
$$\cc C(x,\l)={e^{\l x}\cc \psi(\l) \over \Pi(\l+\gamma_j)}\Bigl(P_2(x,\l)e^{-\l  x}\cc U_{]-\i,-x[}(\l)-Q_2(x,\l)  \Bigr )$$
$\cc C(x,\l)$ is holomorphic on $ \{Re(\l<0\}$. Then for all $j=1,\dots, n$, 
$${\partial ^{n_j}\over (\partial \l)^{n_j}}\Bigl[{P_2(x,\l)e^{-\l x}\cc U_{]-\i,-x[}(\l)-Q_2(x,\l)\over  \Pi_1^m(\l+\beta_i)}\Bigr]_{\l=-\gamma_j}=0$$

In this last equation, when replacing  $P_2(x,\l)$ and $Q_2(x,\l)$  by their above expressions, we obtain  the next system, for $j=1,\dots, n$,

$${\partial ^{n_j}\over (\partial \l)^{n_j}} \Bigl[e^{-\l x}\cc  U_{]-\i,-x[}(\l)\Bigr]_{\l=-\gamma_j}$$
$$+ \sum_{i=1}^m \cc a_i(x){\partial ^{n_j}\over (\partial \l)^{n_j}}{\partial ^{m_i}\over (-\partial \beta)^{m_i}} \Bigl[{e^{-\l x}\cc  U_{]-\i,-x[}(\l)+ e^{\beta x}\cc U_{[-x,0]}(-\beta)\over \l+\beta}\Bigr]_{\beta=\beta_i,\l=-\gamma_j}=0$$

In other words , 

\begin{equation}\label{stm} \sum_{i=1}^m \cc a_i(x) {\bf W}_{i,j}(x)={\bf v}_j(x)\end{equation}

This  a system of $n$ equations where the variables  are the $m$ terms $\cc a_i(x)$.

When $m=n+1$, we have  the property  of lemma \ref{q2} ${Q_2(x,\l)\over \l^n} \to -\cc u(-x)$ when $\l\to-\i$ then 
$$\sum_{i=1}^m \cc a_i(x){\partial ^{m_i}\over(-\partial\beta)^{m_i}} \Bigl [e^{\beta x}\cc U_{[-x,0]}(-\beta)\Bigr]_{\beta=\beta_i}=\cc u (-x)$$

In other words 
\begin{equation}\label{stmm}  \sum_{i=1}^m \cc a_i(x){\bf w}_i(x)=\cc u (-x)\end{equation}

We have already seen in the previous proof, that $\det {\bf W}(x)\not=0$ if $m=n$ and $\det \tilde{\bf  W} (x)\not=0$  if $m=n+1$. Then the system (\ref {stm}) augmented by the equation (\ref{stmm}) if $m=n+1$, is a 
 Cramer system;  and the solution is  given as follows. If $m=n$,

$$\cc a_i(x)={|{\bf W}_i^{[{\bf v}(x)]} (x)|\over |{\bf W}(x)|}$$
and 
$$\cc A(x,\l)=\cc U_{[-x,0]}(\l)-\sum_1^m \cc a_i(x){\bf m}_i(x,\l)={ \left |\begin{array} {cc} \cc U_{[-x,0]}(\l)  & {\bf v}(x)\\ 
{\bf  m}(x,\l) &  {\bf  W}(x) \end{array}\right |  \over |{\bf W}(x)|}$$

If $m=n+1$, 
$$\cc a_i(x)={|{\bf \tilde W} (x) ]_i ^{[{\cc u}(-x), {\bf v}(x) ]}|\over |{\bf \tilde W}(x)|}$$ 

$$\cc A(x,\l)=\cc U_{[-x,0]}(\l)-\sum_1^m \cc a_i(x)e^{\l x}{\bf m}_i(x,\l)= {\left |\begin{array} {ccc} \cc U_{[-x,0]}(\l)  & \cc u(-x) & {\bf v}(x) \\ 
e^{\l x}{\bf  m}(x,\l) &   {\bf w} (x) & {\bf W}(x) \end{array}\right |  \over |{\bf \tilde W} (x)|}$$

\end{proof}

{\bf Proof of theorem \ref{AccA} and corollary \ref{AccAcor}}

Theorem \ref{AccA} follows from  ropositions  \ref{funct}, \ref{calA} and \ref{calcA}.

After that, corollary  \ref{AccAcor} becomes  obvious by the Laplace inversion of the functions ${\bf m}_i(x,\l)$ given in proposition \ref 
{inf}.

\section{Computation of $\cc C(x,\l)$ and $B(x,\l)$ and proof of theorems \ref{BccC} and corollary \ref {BccCco}}

\begin{proposition} 
$$\cc C(x,\l) =\cc\psi(\l) {\left |\begin{array} {cc} 
\cc U_{]-\i,-x[}(\l)   & -e^{\l x}{\bf n} (x,\l)\\ {\bf w}(x)  & {\bf W}(x)
\end{array} \right |\over |{\bf W}(x)|}$$
$$ B(x,\l)={\left |\begin{array} {cc}  1 &  -{\bf  i}(\l)\\  {\bf w}(x) & {\bf W}(x)
\end{array} \right | \over |{\bf W}(x)|}\quad \hbox{if}\quad  m=n$$

$$\cc C(x,\l)=\cc\psi(\l){\left |\begin{array}{ccc}\l  \cc U_{]-\i, -x[}(\l) &- \cc  U_{]-\i, -x[}(\l) & e^{\l x} {\bf n}(x,\l)\\
 {\bf  w}'(x)  & {\bf w}(x) & {\bf W}(x)\end{array}\right |\over |{\bf \tilde W} (x) | }$$
 $$
B(x,\l)={\left |\begin{array}{ccc}\l & -1 &{\bf i}(\l)\\ {\bf  w'}(x)  & {\bf w}(x) & {\bf W}(x)\end{array}\right |\over |{\bf \tilde  W}(x)  |}\quad \hbox{if}\quad  m=n+1$$

\end{proposition}

\begin{proof} We take the expression of  $\cc C(x,\l)$ of proposition  \ref{poly}, 

 $$\cc C(x,\l)={\cc \psi(\l) \over \Pi_1^n(\l+\gamma_j)}\Bigl(-Q_2(x,\l)e^{\l  x} + P_2(x,\l)\cc U_{]-\i,-x[}(\l)\Bigr)$$
 
with  $\deg P_2=m$ and $\deg Q_2\leq m-1$.

 Because $\cc C(x,\l)$ is holomorphic on $\{\Re(\l)<0\}$,  the function    $e^{-\l x}\cc C(x,\l)$ is in   the next form,
  $$e^{-\l x}\cc C(x,\l)=\cc \psi(\l).\Bigl(\cc c_{-1}(x) \l e^{-\l x}  \cc U_{]-\i,-x[}(\l) + \cc c_0(x)e^{-\l x}\cc U_{]-\i,-x[}(\l) + \cc d_0(x)$$
 $$-\sum_{j=1}^n \cc c_j(x)
 {\partial^{n_j}\over (-\partial \gamma)^{n_j}}\Bigl[ {e^{-\l x} \cc U_{]-\i,-x[}(\l)-  e^{\gamma x}\cc U_{]-\i,-x[}(-\gamma)\over \l+\gamma}\Bigr]_{\gamma=\gamma_j}$$ 
 
 In other words, 
 
 $$e^{-\l x}\cc C(x,\l)$$
 $$= \cc \psi(\l)\Bigl(\cc c_{-1}(x) \l. e^{-\l x}\cc U_{]-\i,-x[}(\l) + \cc c_0(x)e^{-\l x}\cc U_{]-\i,-x[}(\l) +
  \cc d_0(x)  -  \sum_{j=1}^n \cc c_j(x){\bf n}_j(x,\l)\Bigr)$$
 for some coefficients $\cc c_j(x)$, $j=-1,0,1,\dots, n$ and a coefficient $\cc d_0(x)$.
 
Thus  $$P_2(x,\l)=\Pi(\l+\gamma_j)\Bigl(\cc c_{-1}(x) \l+\cc c_0(x)-\sum_{j=1}^n \cc c_j(x)
 {\partial^{n_j}\over (-\partial \gamma)^{n_j}}\Bigl[{ 1 \over \l+\gamma}\Bigr]_{\gamma=\gamma_j}\Bigr)  $$
And 
$$Q_2(x,\l)=\Pi(\l+\gamma_j)\Bigl(-\cc d_0(x)-\sum_{j=1}^n \cc c_j(x) {\partial^{n_j}\over (-\partial \gamma)^{n_j}}\Bigl[{e^{\gamma x}\cc U_{]-\i,-x[}(-\gamma ) \over \l+\gamma})\Bigr]_{\gamma=\gamma_j}\Bigr)$$

Take the equation  of proposition \ref {funct} $$e^{-\l x}\cc A(x,\l)={1\over \Pi(\l+\beta_i)}\Bigl(Q_2(x,\l) + P_2(x,\l)e^{-\l  x}\cc U_{[-x,0]}(\l)\Bigr).$$

The function $\cc A(x,\l)$ is  an entire function, then we obtain the equation for every $i=1,\dots, m$,

$${\partial ^{m_i}\over (\partial \l)^{m_i}}\Bigl[{Q_2(x,\l) +P_2(x,\l)e^{-\l x}\cc U_{[-x,0]}(\l)\over \Pi_1^n(\l+\gamma_j)}\Bigr]_{\l=-\beta_i}=0$$

Taking the expressions of $P_2$ and $Q_2$ upwards, we obtain the next equations,  

$${\partial ^{m_i}\over (\partial \l)^{m_i}}  \Bigl[\l \cc c_{-1} (x) e^{-\l x}\cc U_{[-x,0]}(\l)+\cc c_0(x) e^{-\l x}\cc U_{[-x,0]}(\l)- \cc d_0(x)\Bigr]_{\l=-\beta_i}  $$

$$- \sum_{j=1}^n \cc c_j(x)\Bigl[{\partial ^{m_i}\over (\partial \l)^{m_i}} {\partial^{n_j}\over (-\partial \gamma)^{n_j}}  {e^{-\l x}\cc U_{[-x,0]}(\l) +e^{\gamma x} \cc U_{]-\i,-x[}(\gamma)\over \l+\gamma}\Bigr]_{\l=-\beta_i,\gamma=\gamma_j}=0 $$

In other words, 
\begin{equation} \label{sys1} \cc c_0(x) {\bf w}_i(x)+ \sum_{j=1}^n \cc c_j(x) W_{i,j}(x)= {\partial ^{m_i}\over (\partial \l)^{m_i}}  \Bigl[ -\cc c_{-1}(x)\l e^{-\l x}\cc U_{[-x,0]}(\l) +\cc d_0(x)\Bigr]_{\l=-\beta_i} \end{equation}

We can compute the coefficients $\cc c_{-1}(x)$, $\cc c_0(x)$ and $\cc d_0(x)$ by using  the assertions of lemma \ref {poly}.
More precisely, if $m=n$, we have    $P_2(x,\l)\sim \l^n$ and  $\deg Q_2\leq n-1$, thus we obtain 

\begin{equation} \label{hqfa} \cc c_{-1}(x)=\cc d_0(x) =0\qquad \cc c_0(x)=1 \end{equation}

If $m=n+1$ then    $P_2(x,\l)\sim \l^m$ and  according to proposition \ref{q2}, $Q_2(x,\l) \sim [-\cc u(-x)] \l^n$, thus we obtain

\begin{equation} \label{hqfb}  d_0(x)=\cc u(-x) \qquad  \cc c_{-1}(x)=1\end{equation}

Then, if $m=n$,   using  the system (\ref {sys1}) and equation (\ref {hqfa}) we get the new system 

$$ \sum_{j=1}^n \cc c_j(x) {\bf W}_{i,j}(x)=-{\bf w}_i(x)\qquad \hbox{ for } i=1,\dots,n$$
Thus 

 $$\cc c_j(x)=-{ | {\bf W} ^j_{[{\bf w}(x)]}(x)|\over |{\bf W}(x)|}$$

And $$\cc C(x,\l) =\cc \psi(\l).\Bigl(\cc U_{]-\i,-x[}(\l)- e^{\l x}.\sum_{j=1}^n \cc c_j(x){\bf n}_j(x,\l)\Bigr)$$
$$=\cc\psi(\l) {\left |\begin{array} {cc} 
\cc U_{]-\i,-x[}(\l)   & -e^{\l x}{\bf n} (x,\l)\\ {\bf w}(x)  & {\bf W}(x)
\end{array} \right |\over |{\bf W}(x)|}$$

And $$B(x,\l)={P_2(x,\l)\over \Pi(\l+\gamma_j)} =1 - \sum_1^n \cc c_j(x){\partial ^{n_j}\over (-\partial \gamma)^{n_j}}\bigl[ {1\over \l+\gamma}\bigr]_{\gamma=\gamma_j}=1 - \sum_1^n \cc c_j(x){\bf i}_j(\l)$$

$$={\left |\begin{array} {cc}  1 &  -{\bf  i}(\l)\\  {\bf w}(x) & {\bf W}(x)
\end{array} \right | \over |{\bf W}(x)|}$$

Similarely if $m=n+1$, we obtain from system (\ref {sys1})  and  equation (\ref {hqfb}), the new Cramer system

$$\cc c_0(x) {\bf w}_i(x) +\sum_{j=1}^n  c_j(x) {\bf W}_{i,j}(x)={\partial ^ {m_i}\over (-\partial \beta)^{m_i}}\Bigl[\beta e^{\beta x}\cc U_{[-x,0]}(-\beta) +\cc u(-x)\Bigr]_{\beta=\beta_i}={\bf w} '_i(x)$$
Then (for $j=0,1,\dots,n$), 

$$\cc c_j(x)={\left | {\bf \tilde W} (x)^{j+1}_{[{\bf  w'}(x)]}\right | \over \left |\tilde {\bf W}(x)\right |}$$

And $$\cc C(x,\l)= \cc \psi(\l).\Bigl( \l.\cc U_{]-\i,-x[}(\l) +  \cc u(-x) e^{\l x} + \cc c_0(x)\cc U_{]-\i,-x[}(\l) 
 -  \sum_{j=1}^n \cc c_j(x)e^{\l x}{\bf n}_j(x,\l)\Bigr)$$
 
 $$= {\left |\begin{array}{ccc}\l  \cc U_{]-\i, -x[}(\l) +\cc u(-x)e^{\l x} & -\cc  U_{]-\i, -x[}(\l) &e^{\l x}{\bf n}(x,\l)\\
 {\bf  w}'(x)  & {\bf w}(x) & {\bf W}(x)\end{array}\right |\over\left  | {\bf \tilde W}(x)\right | } $$
 
 $$B(x,\l)={P_2(x,\l)\over \Pi(\l+\gamma_j)} =\l+ \cc c_0(x)-\sum_1^n \cc c_j(x){\partial ^{n_j}\over (-\partial \gamma)^{n_j}}\Bigl[{1\over \l+\gamma}\Bigr]_{\gamma=\gamma_j}$$
 $$=\l+\cc c_0(x) - \sum_{j=1}^n \cc c_j(x){\bf i}_j(\l)$$

$$={\left |\begin{array}{ccc}\l & -1 & {\bf i}(\l)\\ {\bf  w'}(x)  & {\bf w}(x) & {\bf W}(x)\end{array}\right |\over \left |{\bf \tilde  W}(x) \right |}$$
\end{proof}

{\bf Proof of theorem \ref {BccC} and corollary \ref {BccCco}.}

The proof of theorem \ref{BccC} follows from the  above proposition and proposition \ref{funct}. Corollary \ref {BccCco} follow then   with the help  the inverse Laplace transform of the functions  $n_j(x,\l)$ given  in proposition \ref{inf}.

\section{Computation of the functions $C(x,\l)$ and $\cc B(x,\l)$ and proof of theorem \ref {CccB} and corollary \ref {CccBco}}

\begin{proposition}
$$C(x,\l)=e^{-\l x}{\left |\begin{array} {cc} 0 & {\bf  i}(\l)  \\   {\bf c} & {\bf W}(x) 
\end{array} \right | \over |{\bf W}(x)|}
\qquad \cc B(x,\l)=
\cc \psi(\l){\left |\begin{array} {cc} 1& {\bf  n}(x,\l)
 \\   {\bf c} &  {\bf W}(x)  \end{array} \right | \over |{\bf W}(x)|}\quad\hbox{if}\quad m=n $$

$$C(x,\l)= e^{-\l x}
{\left |\begin{array} {ccc} 0 &   1 & {\bf  i}(\l) \\   {\bf c}& {\bf w}(x) &{\bf  W}(x)   
\end{array} \right | 
\over | {\bf \tilde  W} (x)|} $$

$$\cc B(x,\l)=\cc \psi(\l){\left |\begin{array} {ccc} 0 & e^{-\l x} \cc U_{]-\i,-x[}(\l) &{\bf  n}(x,\l) \\   {\bf c} & {\bf w}(x) & {\bf W} (x)   
\end{array} \right | \over |{\bf \tilde  W} (x)|}\quad\hbox{if}\quad m=n+1 $$

\end{proposition}

\begin{proof}

According to proposition \ref{poly}, we have the identity 
$$\cc B(x,\l)={\cc \psi(\l)\over \Pi_1^n(\l+\gamma_j)}.\Bigl(P_1(x,\l) - Q_1(x,\l)e^{-\l x}\cc U_{]-\i,-x[}(\l)\Bigr)$$

with $$\deg P_1(\l)\sim\l^n \quad \hbox{and} \quad \deg Q_1\leq m-1$$
 
Since $\cc B(x,\l)$ is holomorphic on $\{\Re(\l)<0\}$, we can write 
 $\cc B(x,\l)$ in  the following form,
  $$\cc B(x,\l)=\cc \psi(\l)\Biggl (1+\cc b_0(x)e^{-\l x}\cc U_{]-\i,-x[}(\l)+\sum_{j=1}^n \cc b_j(x){\partial^{n_j}\over (-\partial \gamma)^{n_j}}\Bigl[ {e^{\gamma x}\cc U_{]-\i,-x[}(-\gamma)- e^{-\l x}\cc U_{]-\i,-x[}(\l)\over \l+\gamma}\Bigr]_{\gamma=\gamma_j}\Biggr)$$
 
 In other words, 
 
 $$\cc B(x,\l)= \cc \psi(\l).\Bigl(1+\cc b_0(x)e^{-\l x}\cc U_{]-\i,-x[}(\l) -\sum_{j=1}^n \cc b_j(x){\bf n}_j(x,\l)\Bigr)$$
 for some coefficients $\cc b_j(x)$, $j=-1,0,1,\dots, n$ and a coefficient $\cc d_0(x)$.

Thus the polynomial $P_1$ and $Q_1$ are given by the next expressions

$$P_1(x,\l)=\Pi_1^n(\l+\gamma_j)\Biggl(1+\sum_{j=1}^n \cc b_j(x){\partial ^{n_j}\over (-\partial \gamma)^{n_j}}\Bigl[ {e^{\gamma x}\cc U_{]-\i,-x[}(-\gamma)\over \l+\gamma})\Bigr]_{\gamma=\gamma_j}\Biggr)$$

$$Q_1(x,\l)=\Pi_1^n(\l+\gamma_j)\Biggl(-\cc b_0(x)+\sum_{j=1}^n \cc b_j(x){\partial ^{n_j}\over (-\partial \gamma)^{n_j}}\Bigl[ {1 \over \l+\gamma}\Bigr]_{\gamma=\gamma_j}\Biggr)$$

On the other hand,  according to proposition \ref{funct}, we have 

$$A(x,\l) ={1\over \Pi_1^m(\l+\beta_i)}
\Bigl(P_1(x,\l)+ Q_1(x,\l)e^{-\l x}\cc U_{[-x,0]}(\l)\Bigr)$$

Since $A(x,\l)$ is an entire function of $\l$, then it is holomorphic at  $\l=-\beta_i$; Thus we have the equations for $i=1,\dots, m$.
 $${\partial ^{m_i}\over (\partial\l)^{m_i}}\Bigl[{P_1(x,\l)+ Q_1(x,\l)e^{-\l x}\cc U_{[-x,0]}(\l)\over  \Pi_1^n(\l+\gamma_j)}\Bigl]_{\l=-\beta_i}=0,$$ 
that is,  
$$\cc b_0(x){\partial ^{m_i}\over (\partial \l)^{m_i} }\Bigl[e^{-\l x} \cc U_{[-x,0]}(\l)\Bigr]_{\l=-\beta_i}$$
$$-\sum_{j=1}^n \cc b_j(x){\partial ^{m_i}\partial ^{n_j}\over (\partial\l)^{m_i} (-\partial\gamma)^{n_j}}\Bigl[{e^{\gamma x}\cc U_{]-\i,-x[}(-\gamma)+e^{-\l x}\cc U_{[-x,0]}(\l) \over \lambda+\gamma }\Bigr]_{\gamma=\gamma_j,\l=-\beta_i}=1_{m_i=0}$$

In other words, 

\begin{equation}\label {stmbcc}
\cc b_0(x){\bf w}_i(x) + \sum_{j=1}^n \cc b_j(x){\bf W}_{i,j}(x)=1_{m_i=0}
\end{equation}

Moreover, if $m=n$, since $\deg Q_1(x,\l)\leq n-1$, we have

\begin{equation}\label {stmbcca} \cc b_0(x)=0\end{equation}

We have already seen that the determinants $|{\bf W}(x)|$ ( if $m=n$) and $|{\bf \tilde W}(x)|$  (if $m=n+1$), then the system (\ref{stmbcc}) augmented with 
(\ref{stmbcca}) if $m=n$, form a Cramer system whose  solutions are  given by the next expressions,

If $m=n$, 

$$\cc  b_0(x)=0\qquad \cc b_j(x)={|{\bf W}_j^{\bf c}(x)|\over |{\bf W}(x)|}\quad \hbox{ for }j=1,\dots,n$$
Thus 

$$\cc B(x,\l)=\cc \psi(\l)(1-\sum_1^n  \cc b_j(x){\bf n}_j(x,\l))=
\cc \psi(\l){\left |\begin{array} {cc} 1 & {\bf  n}(x,\l)
 \\  {\bf  c} &  {\bf W}(x)  \end{array} \right | \over |{\bf W}(x)|}$$

and the expression of $C(x,\l)$ given in proposition \ref{poly} gives

$$C(x,\l)=-{Q_1(x,\l) e^{-\l x}\over \Pi_1^n(\l+\gamma_j)}=e^{-\l x}(\sum_1^n  {-\cc b_j(x)\over \l+\gamma_j})=e^{-\l x}{\left |\begin{array} {cc} 0 & {\bf  i}(\l)  \\   {\bf c} & {\bf W}(x) 
\end{array} \right | \over |{\bf W} (x)|}$$

If $m=n+1$, we have for $j=0,\dots, n$, 
$$\cc b_j(x)={\left | {\bf \tilde W}_{j+1}^{\bf c}(x)\right |\over \left |{\bf \tilde W}(x) \right |}$$

And $$\cc B(x,\l)=\cc \psi(\l)\Bigl (1-\cc b_0(x)e^{-\l x}\cc U_{]-\i,-x[}(\l) -\sum_1^n \cc b_j(x){\bf n}_j(x,\l)\Bigr )$$
$$=
\cc \psi(\l){\left |\begin{array} {ccc} 1  &  e^{-\l x}\cc U_{]-\i,-x[}(\l)  & {\bf  n}(x,\l) \\   {\bf c} & {\bf w}(x) &{\bf W}(x)   
\end{array} \right | \over \left |{\bf \tilde W}(x) \right |}$$

and   then,  

$$C(x,\l)=-{Q_1(x,\l) e^{-\l x}\over \Pi_1^n(\l+\gamma_j)}=e^{-\l x}(\cc b_0(x) - \sum_1^n {\cc b_j(x)\over \l+\gamma_j})$$
$$= e^{-\l x}{\left |\begin{array} {ccc}   0 & 1 & {\bf i}(\l)  \\   {\bf c} &  {\bf w}(x) &{\bf W}(x)   
\end{array} \right | \over \left |{\bf \tilde W}(x) \right |}$$
\end{proof}

{\bf Proof of corollary \ref{CccBco}}. This corollary becomes obvious when using propositions \ref{funct} and the previous one and when using the Laplace thransform inversion of ${\bf n}_j(x,\l)$ and ${\bf i}_j(\l)$ given in proposition \ref{inf}

\section{Proof of theorems \ref{Ux} and \ref{Tab}}

It has been proven in \cite{F10} (propositions 6.3. and 6.4) that the distributions of the triple $(X,I,S)$ at time $U_x$ and at time $T_a^b$ are also characterized by the $6$ functions $A,\cc A, B,\cc B, C,\cc C$. Then, we  translate these two propositions with  reference to the explicit forms of the functions obtained previously.

\bigskip

{\it Acknowledgment}. this work has been supported by the ANR, on the project  ANR-09-BLAN-0084-01.
\bibliographystyle{amsplain}

\end{document}